\documentclass[11pt,reqno]{amsart}

\usepackage{color}

\newtheorem{theorem}{Theorem}[section]
\newtheorem{lemma}[theorem]{Lemma}
\newtheorem{prop}[theorem]{Proposition}
\newtheorem{cor}[theorem]{Corollary}

\newtheorem{rem}[theorem]{Remark}

\theoremstyle{definition}
\newtheorem{defin}[theorem]{Definition}
\newtheorem{slem}[theorem]{{\sc Sublemma}}
\newtheorem{remark}[theorem]{\sc Remark}

\numberwithin{equation}{section}

\usepackage{mathrsfs, amsmath,amssymb,amsfonts, amsthm, dsfont}
\usepackage{hyperref}

\newcommand{\C}{{\mathbb C}}

\newcommand{\la}{\lambda}

\newcommand{\e}{\varepsilon}

\newcommand{\supp}{\text{supp }}
\renewcommand{\Pi}{\varPi}
\renewcommand{\Re}{\rm{Re} \,}

\renewcommand{\epsilon}{\varepsilon}

\newcommand{\hcal}{{\mathcal H}}
\newcommand{\fcal}{{\mathcal F}}

\numberwithin{equation}{section}

\newcommand{\R}{{\mathbb R}}

\usepackage{epsfig}
\usepackage{amscd}
\usepackage{amsmath, amssymb, amsthm}
\usepackage{graphics}
\usepackage{color}
\usepackage{dsfont}
\newtheorem*{main-theorem}{Main Theorem}
\newtheorem*{old-thm}{Theorem}
\theoremstyle{definition}
\numberwithin{equation}{section}

\def\11{\mathds{1}}

\def\Re{\,\mathrm{Re}\,}

\def\supp{\mathrm{supp}\,}

\def\phi{\varphi}
\def\half{{\frac{1}{2}}}

\def\be{\begin{eqnarray*}}
\def\ee{\end{eqnarray*}}
\def\ben{\begin{eqnarray}}
\def\een{\end{eqnarray}}

\def\L2R{L_{\text{Rest}}^2}

\newcommand{\lcal}{\mathcal{L}}

\newcommand{\ncal}{\mathcal{N}}

\newcommand{\ocal}{\mathcal{O}}

\newcommand{\scal}{\mathcal{S}}

\begin{document}
\title[Sup norms of Cauchy data ]{Sup norms of Cauchy data of eigenfunctions on 
manifolds with concave boundary}

\author{Christopher D. Sogge  and Steve Zelditch}

\address{Department of Mathematics, Johns Hopkins University, Baltimore,
MD, 21218}

\address{Department of Mathematics, Northwestern  University, Evanston, IL 60208, USA}

\email{csogge@math.jhu.edu}

\email{zelditch@math.northwestern.edu}

\thanks{Research partially supported by NSF grants DMS-1361476, resp. DMS-1206527 and DMS-1541126.}
\begin{abstract}
We prove that the Cauchy data of Dirichlet or Neumann   $\Delta$- eigenfunctions of  Riemannian manifolds with concave (diffractive) boundary can only
achieve maximal sup norm bounds if there exists a self-focal point on the boundary, i.e. a point at which a positive
measure of geodesics leaving the point return to the point. In the case of real analytic Riemannian manifolds with
real analytic boundary, maximal sup norm bounds on boundary traces of eigenfunctions can  only be achieved if there
exists a point on the boundary at which all geodesics loop back. As an application, the Dirichlet or Neumann 
eigenfunctions of Riemannian manifolds with
concave boundary and non-positive curvature never have eigenfunctions whose boundary traces achieve maximal sup norm bounds. The key new
ingredient is the Melrose-Taylor diffractive parametrix and Melrose's analysis of the Weyl law.

\end{abstract}
\maketitle

\section{Introduction}

Let $(X, g)$ be a compact  Riemannian manifold of dimension $n$, and let $M = X \backslash \ocal$ be the 
exterior of a finite disjoint union of open convex sub-domains (`obstacles').  Thus, $(M, g)$ is a Riemannian manifold  with
geodesically   {\it concave}  boundary $\partial M$.   We consider the eigenvalue problem,
$$\left\{ \begin{array}{l}  -\Delta \phi_{\lambda} = \lambda^2 \phi_{\lambda},
\\
B \phi_{\lambda}  = 0 \;\; \mbox{on}\;\; \partial M \end{array} \right.,$$ where $B$ is the boundary operator,
either  $B \phi = \phi|_{\partial M}$ in the Dirichlet case or $B
\phi = \partial_{\nu} \phi|_{\partial M}$ in the Neumann case.     We  denote by $\{\phi_{j}\}$
an orthonormal basis of eigenfunctions,  $ \langle \phi_j, \phi_k \rangle = \delta_{jk}$, with
$ \lambda_1 < \lambda_2 \leq \cdots $  counted with multiplicity. The inner product is
defined by   $\langle f, g \rangle = \int_{M} f \bar{g} dA$ where $dA$ is the area form of $g$. 

This article is  concerned with sup norm bounds on the Cauchy data 
$$(\phi_j |_{\partial M}, \lambda_j^{-1} \partial_{\nu} \phi_j |_{\partial M} )$$
of Dirichlet  (resp.  Neumann)
eigenfunctions along the boundary. The non-trivial component of the  Cauchy data of eigenfunctions satisfying boundary conditions is often
called the `boundary trace'. We   denote by $r_q u$ the restriction of $u \in C(\bar{M})$ to $\partial M$ at the point
$q \in \partial M$, and we 
denote by  $\gamma_q^B$ the  boundary trace with boundary conditions $B$. Thus, \begin{equation}\label{SCTr} \gamma_q^B =  \left\{ \begin{array}{ll}   r_q, & \rm{Neumann \;case} \\ &\\
r_q \partial_{\nu_q}, &  \rm{Dirichlet \;case} \end{array} \right. \end{equation}  We also denote by
  \begin{equation} \label{phib} \phi_j^b(q) = \gamma_q^B  \phi_j \end{equation}
the non-trivial boundary trace of an eigenfunction satisfying the boundary condition $B$ (either Dirichlet or Neumann). \footnote{
It is often natural to define the normal derivative as the semi-classical derivative  $ \lambda^{-1} 
\partial_{\nu_q}$ and to define the boundary trace in the Dirichlet case
by $ \lambda^{-1} 
r_q \partial_{\nu_q}$ as in (1.2)or Theorem 6  of \cite{HHHZ}. In this article we do not use the semi-classical normal derivative because it would
complicate the formula for the boundary trace of the wave group. }

The goal of this article, as in \cite{SZ1,STZ,SZ2},  is first  to prove universal sup norm bounds on boundary traces
of eigenfunctions (Corollary \ref{REMJUMP}) and more significantly  to characterize the Riemannian manifolds with concave boundary where the sup norm bounds
are achieved. When the universal bounds are achieved, we say that $(M, g, \partial M)$ has {\it maximal eigenfunction
growth along the boundary. }  The sup norm  problem on Cauchy data   makes sense for any Riemannian manifold
with boundary,  but
we restrict to the case of  concave boundaries in this article. Our main result, Theorem \ref{SoZthm}, states
that maximal sup norm bounds are never achieved on Riemannian manifolds with concave boundary and without boundary
self-focal points. In particular, this applies to non-positively curved Riemannian manifolds with concave boundary.  Further motivation to study non-positively curved manifolds with
concave boundary is that
the sup norm results are   needed in \cite{JZ} to count nodal domains on non-positively curved surfaces with concave boundary.

\subsection{Statement of results}

 Before stating the results,
we recall what is known about global sup norm bounds when $\partial M = \emptyset$ and also in the
interior of $M$ when $\partial M \not= \emptyset$.
 The `universal' sup norm bound
\begin{equation} \label{STANDARD} ||\phi_j||_{L^{\infty}} \leq C_g\; \lambda_j^{\frac{n-1}{2}} \end{equation} on n-dimensional Riemannian manifolds  without boundary was extended to manifolds with boundary in \cite{Gr,Sm,Sog}
 But this bound is rarely achieved, and in the articles
\cite{SZ1,STZ,SZ2} there are successively stronger constraints on Riemannian manifolds
without boundary for which  the bound \eqref{STANDARD} is achieved. We refer to such Riemannian manifolds  `$(M,g)$
as having maximum eigenfunction growth', and express the condition as
\begin{equation} \label{OMEGA} ||\phi_j||_{L^{\infty}}  = \Omega( \lambda_j^{\frac{n-1}{2}}), \end{equation}
where $\Omega$ is the negation of ``little oh", i.e. the $\Omega$-symbol indicates that there exists some
subsequence of eigenfunctions $\phi_j$ with $\lambda_j \to \infty$ such that the standard
bound is achieved.
In the boundary-less case of \cite{SZ1}, it is proved that a necessary    condition for maximal growth \eqref{OMEGA} is that
there exists a `self-focal point' or `partial blow down point'  $p$. We will define the terms below. The result was
improved in \cite{STZ} and further improved in \cite{SZ2}, but we will only be concerned with the original condition
in this article. 

\subsubsection{Universal sup norm bounds on manifolds with concave boundary}

The first result of this article gives universal sup norm bounds on Cauchy data parallel to those of \cite{Gr,Sm,Sog} 
in the interior. Recalling \eqref{phib}, we  denote by 
\begin{equation} \label{BTSPEC} \Pi^b_{[0,\lambda]}(q,q') = \sum_{j: \lambda_j \leq \lambda} \phi_j^b(q) \phi_j^b(q') \end{equation} the boundary trace of the spectral projection for $\sqrt{-\Delta}$ for the interval $[0, \lambda]$.

\begin{prop} \label{PTWEYL} For any $C^{\infty}$ Riemannian manifold $(M, g)$ of dimension n with $C^{\infty}$ concave boundary $\partial M$,
$$\Pi^b_{[0,\lambda]}(q,q)= \left\{\begin{array}{ll} C_n \lambda^{n +2 } 
+  \lambda^2
R_D^b(\lambda, q) , & \mbox{Dirichlet}
\\ & \\
C_n \lambda^n 
+ R_N^b(\lambda, q), & \mbox{Neumann}.
\end{array} \right.$$

with

$$R^b_B(\lambda, q)  = O(\lambda^{n-1}) \;\; \mbox{uniformly in q}. $$

\end{prop}
Complete expansions for smoothed Cauchy data Weyl sums are given in Proposition \ref{SF2}. 
The result
was stated but not proved in \cite{Z4}. 
Universal sup norm bounds  on boundary traces of  eigenfunctions are  obtained from
the jump in the remainder:

\begin{cor}\label{REMJUMP} Under the assumptions above,
\begin{equation}\label{eig3} \sum_{j: \lambda_j = \lambda} |\phi_j^b(q)|^2
=
R_B^b(\lambda,q)-R_B^b(\lambda-0,q) = O(\lambda^{n-1}),
\end{equation}
in the Neumann case and similarly with an extra factor of $\lambda^2$ in the 
Dirichlet case. The remainder is uniform in $q$. Hence, in the Neumann case,
$$\sup_{q \in \partial M} |\phi_j^b(q)| \leq C \lambda^{\frac{n-1}{2}}, \;\; $$
and similarly in the Dirichlet with the right side replaced by $\lambda^{\frac{n+1}{2}}.$
\end{cor}

\subsection{Manifolds with concave boundary and no self-focal boundary points never achieve maximal sup norm bounds}

Our main result is an improvement of the sup-norm estimate of Corollary \ref{REMJUMP} in the
case of manifolds of concave boundary and no self-focal points.

We denote by $\Phi^t$ the  billiard flow (or broken geodesic flow) of $(M, g, \partial M)$. We also denote the broken
exponential map by
$\exp_x \xi = \pi \Phi^1(x, \xi)$. We refer to \cite{HoI-IV} (Chapter XXIV)  or \cite{MT} for background on these notions.

 Given any $x\in \bar{M}$, we denote by  $\lcal_x$ the set
of loop directions at $x$,
\begin{equation} \lcal_x = \{\xi \in S^*_x M : \exists T: \exp_x T
\xi = x \}.
\end{equation}

\begin{defin} We say that   $x$ is a self-focal  point if $|\lcal_x| > 0$ where
$|\cdot|_x$ denotes the surface measure on $S^*_x M$ determined by
the Euclidean metric $g_x$ on $T_x^*M$ induced by $g$. \end{defin}

.






The main result of \cite{SZ1} is that if $(M, g)$ is $C^{\infty}$,  $\partial M = \emptyset$,  and has maximal sup-norm
growth in the sense of \eqref{OMEGA}, then it must possess at least one point $p$
for which $|\lcal_p| > 0$, i.e. a self-focal point. The main  result of this article proves the same result for Cauchy data of
manifolds with concave boundary.

\begin{theorem} \label{SoZthm}
Let $(M, g)$ be a Riemannian manifold of dimension n with geodesically concave boundary. Suppose that
there exist no self-focal points $q \in \partial M$. Then, in the Neumann case,  $$\sup_{q \in \partial M} |\phi_j^b(q)| = o( \lambda^{\frac{n-1}{2}}), \;\; $$
and similarly in the Dirichlet with the right side replaced by $\lambda^{\frac{n+1}{2}}.$
\end{theorem}
\bigskip

The proof of Theorem \ref{SoZthm} is given in \S \ref{SUPSECT}, following
the strategy 
 in \cite{SZ1}. The main step in the proof is the upper bound on smoothed pointwise Weyl sums in  Lemma \ref{1.11}, 
together
with a compactness argument explained at the beginning of Section \ref{SUPSECT}.

Thus, the Cauchy data can only achieve maximal sup norm bounds if there exists a self-focal point on the boundary. 
Examples are given in Sections \ref{EXAMPLEINTRO} and \ref{EXAMPLESECT}. 
Concavity of the boundary is assumed in  order to calculate the singularity at $t = 0$ of the Cauchy data of the wave
group,  and to rule out the possibility of a sequence of eigenfunctions with  strongly enhanced sup norms at the boundary due to diffractive effects.  That is, we need 
to rule out enhanced values of eigenfunctions in  microlocal $\epsilon$-neighborhood around the tangent directions to $\partial M$. For the interior
of convex domains, whispering gallery modes do concentrate around the boundary and peak there. Even so,  whispering gallery modes do not saturate sup norm estimates. For the exterior of convex domains, which is the setting of this article, there do not even exist analogues of whispering gallery modes. But to prove that no sequence of eigenfunctions has strongly enhanced mass due to grazing rays, we  rely here on the Melrose-Taylor parametrices for the wave group.

We refer to the universal bounds as `convexity bounds', and improvements in the case of special geometries
as `sub-convexity bounds'. Our goal is to characterize geometries of $(M, g, \partial M)$ for which sub-convexity
holds.  To state our main result we now introduce the relevant geometric features.
 It appears plausible that
 Theorem \ref{SoZthm} is valid for all Riemannian manifolds with smooth boundary, and possibly for manifolds with corners such as the Bunimovich
stadium or a hyperbolic billiard. But such generalizations would require a different proof which does not make use of a parametrix construction for the wave group of a manifold with boundary. In future
work, we plan to use the scaling method Seeley and Melrose \cite{M84} (in his proof  of Ivrii's two term
Weyl law  \cite{I80}))  in place of a parametrix construction to estimate
sup norms of Cauchy data. However the diffractive case is more concrete and is of independent interest in illustrating
the application of the Melrose-Taylor parametrix, and Melrose's technique
for obtaining the remainder estimate in Weyl's law.

\subsubsection{Manifolds with concave boundary and no self-focal points}

 Theorem \ref{SoZthm}  is based on the estimate in Lemma \ref{1.11} on smooth sums of the Cauchy
 data $|\phi_j^b(q)|^2$.  Although not necessary for the proof of Theorem \ref{SoZthm}, it is natural to 
 ask if the pointwise Weyl law for Cauchy data in Proposition \ref{PTWEYL} can be improved under the
 same no-focal point assumption. 
The next result is a refinement of Proposition \ref{PTWEYL} with the additional assumption
that the set of loops at each $q \in \partial M$ has $(n-1)$-dimensional measure zero.

\begin{prop} \label{EPLEM}  Assume  $(M, g)$ is of dimension n with $C^{\infty}$ concave boundary $\partial M$, and assume that the set of billiard loops at $q \in
\partial M$ has measure zero in $B^*_q\partial M$. Then there exist  continuous functions $Q_D(q)$,
resp. $Q_N(q)$ 
such that

$$\Pi^b_{[0,\lambda]}(q,q)= \left\{\begin{array}{ll} C_n \lambda^{n +2 }  + Q_D(q) \lambda^{n+1} 
+  
R_D^b(\lambda, q) , & \mbox{Dirichlet}
\\ & \\
C_n \lambda^n 
+ Q_N(q) \lambda^{n-1} 
+ R_N^b(\lambda, q), & \mbox{Neumann}.
\end{array} \right.$$
so that, given $\varepsilon>0$,  there exists  a ball $B \subset \partial M$
centered at $q$ and  $\Lambda<\infty$ so that for $\lambda \geq
\Lambda$,

\begin{equation}\label{W1}
 \left\{
\begin{array}{l} |R_D(\lambda,q)|\le \varepsilon\lambda^{n + 1}, \, \, q\in B,\\ \\
|R_N(\lambda,q)|\le \varepsilon\lambda^{n - 1}, \, \, q\in B,
\end{array} \right.
\end{equation}
where $R_D, R_N$ are defined in Proposition \ref{PTWEYL} .

\end{prop}
The remainder estimate implies that $R_D(\lambda, q) = o(\lambda^{n +1})$, resp. $R_N(\lambda, q) = o(\lambda^{n-1})$.
This gives another proof of:

\begin{cor}\label{REMJUMPNSF}  If $(M,g)$ is a compact Riemannian manifold with concave boundary and with no self-focal points,
then 
\begin{equation}\label{eig4} \sum_{j: \lambda_j = \lambda} |\phi_j^b(q)|^2
= o (\lambda^{n-1}),
\end{equation}
in the Neumann case and similarly with an extra factor of $\lambda^2$ in the 
Dirichlet case.
\end{cor}

The functions $Q_D(q)$ and $Q_N(q)$ are conjectured to be given by
\begin{equation} \label{QDN} Q_D(q) = C_{D,n} H(q), \;\; Q_N(q) = C_{N,q} H(q), \end{equation}
where $H(q)$ is the mean curvature of the boundary at $q$ and where $C_{D,n}, C_{N,n}$ are dimensional
constants whose sign depends on the boundary condition. In the Dirichlet case, the coefficients were conjectured by
Ozawa \cite[p. 304]{O} with no assumption on the boundary. Miyazaki \cite{Mi} proved the formula for the interior
of a Euclidean ball.   A  calculation of the first two coefficients at the physics level of rigor is 
given in (4.3) of  \cite{BSS} in dimension two.  Under the non-focal assumption above, the coefficients are the same as those of
the small $t$ expansion of the Cauchy data of the heat kernel and are known to be local geometric invariants. By
an invariance theory argument, it should be possible  to show that the coefficient must be a universal multiple of
the mean curvature. But we are unaware of a reference where this has been proved. We refer to Lemma \ref{eq:BFSSb} 
for further results of this kind and further discussion.

When the non-focal assumption is dropped, the middle term $Q(q) \lambda^{n-1}$ may become more complicated, analogous
to the interior pointwise asymptotics studied by Safarov \cite[Chapter 1.8]{SV}. In this case, it can depend on $\lambda$, hence be  $Q(\lambda, q) \lambda^{n-1}$ where $Q(\lambda,q)$ is bounded but  not necessarily continuous in $\lambda$; it may have jumps. See \cite[Proposition 1.8.16]{SV}
for an example. It does not appear that the analogous results have been proved for Cauchy data. 

 This result may be compared to 
Melrose's  Weyl law with remainder  for $(M,g)$ with $\partial M$ concave and with zero measure of periodic billiard trajectories
   \cite{M}: 
$$N(\lambda) = \# \{j: \lambda_j \leq \lambda\} =  (2 \pi)^{-n} c_n \rm{Vol}(M)\; \lambda^n + c_{n,B} \rm{Vol}(\partial M) \lambda^{n-1} + o(\lambda^{n-1}). $$
Here, $c_{n,B}$ is a coefficient depending on the type (Dirichlet or Neumann) boundary conditions. It is positive
for Neumann boundary conditions and negative for Dirichlet boundary conditions. There is also 
an additional assumption in \cite{M} that we discuss in Section \ref{REMARKON}  below. Melrose used Airy type
glancing parametrices constructed in \cite{M75,M78,  Tay1}. Ivrii proved this result for general manifolds with boundary
\cite{I80}.

\subsection{Sketch of the proofs}

The proofs of  Proposition \ref{PTWEYL} and Proposition \ref{EPLEM}  rely on three results on boundary traces, some of which are valid without the concavity assumption.
The first (and simplest) result consists of upper bounds on the  wave front set of the boundary trace of the wave kernel (\S \ref{WF}, \eqref{WF1} and \eqref{WF2}.)
The second is the study of the singularity at $t = 0$ of the Cauchy data of the wave group. This uses microlocal parametrices
for the wave group and at this time they are only available when $\partial M \not= \emptyset$ in the case of concave boundary.
The third result considers long time singularities of the Cauchy data Weyl sums due to loops at points on the boundary.

\subsubsection{Singularity at $t = 0$ of Cauchy data of the wave group}

The main novelty  in the proof of Theorem \ref{SoZthm} (relative to \cite{SZ1} involve the diffraction effects of
a concave boundary. A key step is to determine the singularity at $t = 0$ of  
\begin{equation}\label{Sft} \begin{array}{lll} S_q(t): & = &  E_B^b(t, q, q)
=  \sum_{j} \cos t \lambda_j \left| \phi^b_j(q) \right|^2, \;\; q \in \partial M.
\end{array} \end{equation}
For the proof of Proposition \ref{OpaPTWEYL} one constructs a similar distribution with the cutoff Cauchy data.
The singularities of  $E_B^b(t, q, q)$ are due to geodesic loops at $q$, and so there is  an isolated singularity at $t = 0$ for any type of boundary.  We  prove that near $t = 0, S_q(t)$ is a classical co-normal (or,  Lagrangian) distribution, i.e.  that the cutoff Fourier transform
$\fcal_{t  \to \lambda} \hat{\rho}(t) S_q(t)$ is a classical symbol when
$\hat{\rho}(t) \in C_0^{\infty}(- \epsilon, \epsilon)$.

\begin{prop} \label{MELROSE} The diagonal boundary trace of the Neumann wave kernel, $E_N^b(t, q, q)$ or Dirichlet
wave kernel $E_D^b(t, q, q)$  has a classical  co-normal singularity
at $t = 0$. In the Neumann case the leading singularity is  of order $t^{-n}$ (as in the boundaryless case) while in the Dirichlet case, it  is of order $t^{-(n +2)}$ (due to the two normal derivatives).  \end{prop}

 The main difficulty  is to prove that when $\partial M$ is concave, the singularity at $t = 0$ of the glancing part of $E_B^b(t,q,q)$ is normal, i.e. $(t D_t)^k E_B^b(t,q,q)$ has the same order of singularity for all $k$.\footnote{We follow the terminology of \cite{I80} in referring to such a singularity as `normal'.} The glancing parametrix consists of a number of terms in \eqref{MICRODECOMP}. 
In the following, we denote them by  $w_{G}$. We refer to Section \ref{PARA} for  definitions and background. 
  The normality of the singularity  is standard  for the hyperbolic part of the wave kernel, so
the main point is to prove the following Lemma:

\begin{lemma} \label{MLEMintro} On the diagonal of $\partial M \times \partial M$, the glancing part $w_{G}$ of the parametrix has a normal singularity at $t = 0$, i.e.
\begin{equation} \label{edef} w_{G}(t, q, q) = \int e^{i t \theta} e(t, \theta, q) d \theta, \end{equation}
where  $e(t, \theta, q)$  is a classical symbol in $\theta$ of order $n-1 $ when $ q \in \partial M$. 

Moreover, if the microlocal cutoff to the glancing set has support in the $\epsilon$-neighborhood of the
glancing set, then   $$e(t, \theta, q) =  c_{\epsilon}(q)\; \theta^{n-1} \; \rm{mod}\; S^{n-2}, \rm{where} \; c_\epsilon(q)=O(\epsilon)  \; \rm{as} \; \epsilon \to 0.$$

 \end{lemma}
 
The last statement says that as $\epsilon \to 0$, the contribution of
grazing rays to the singularity at $t = 0$ decays to zero, so that, dually
there is no enhanced mass of eigenfunctions in grazing directions (as stated in Proposition \ref{OpaPTWEYL}). The proof of Lemma \ref{MLEMintro} is given
in Section \ref{GLSING}.

 It  seems `intuitive'  that the singularity 
at $t = 0$ should be of a  very simple type, viz. a co-normal one. However,
 `the geometric theory of diffraction' does not provide much intuition about
the singularity \cite{LK59}. One may imagine that the singularity at $t = 0$ has only an infinitesimal period of time to reflect the  `creeping' nature of the waves and  existence of diffractive rays.  But diffractive rays do have
a serious technical impact;  as is shown in Section \ref{GLSING}, the diagonal of the  glancing part of the wave group makes a non-zero contribution involving Airy functions when restricted to the boundary. 

The proof of Lemma \ref{MLEMintro} and Proposition \ref{MELROSE}
uses
the massive machinery of the  Melrose-Taylor diffractive parametrix.
The parametrix does simplify considerably when restricted to the boundary, 
but the Airy factors remain. 

It may be helpful to explain in advance what we use from \cite{M} and how it relates to Melrose's proof of the small $o(\lambda^{n-1})$ remainder estimate for the Weyl law. 
The proof of Lemma \ref{MLEMintro} and Proposition \ref{MELROSE}  uses   the  parametrix for the wave
group in the case of  concave boundary \cite{M, M75,MT,Tay1,Tay2}. We adapt 
Melrose's proof of the Weyl law for manifolds with concave
boundary to obtain the pointwise Weyl laws. In fact, they
are  simpler than the integrated  Weyl laws  of \cite{M}.
The key calculation in Lemma \ref{MLEMintro} (see also Lemma \ref{MLEM}) is based on a modification of Melrose's ingenious proof of normality of the singularity at $t = 0$ in the interior case.  
Thus, a good deal of the content of Sections \ref{PARA} and \ref{GLSING}  consists in reviewing 
the relevant  results and constructions  of  \cite{M}. Indeed, our original idea was that it should be simple to extract from \cite{M}  the remainder estimates we need in  the asymptotics
of $ \Pi^b_{[0,\lambda]}(q,q)$.   In practice, the extraction requires a substantial  amount of material from \cite{M} and implicitly from the 
background article \cite{M75,Tay1} that we have to review in this
article.

\subsection{Examples of non-positively curved  surfaces   with concave boundary}

A special but important case of manifolds with concave boundary are 
 those of  non-positive curvature, obtained by removing some disjoint convex `obstacles' from  a non-postively curved manifold  $X$.  Thus,
\begin{equation} \label{M} M= X \backslash \bigcup_{j = 1}^r \ocal_j,  \end{equation}
where $\ocal_j$ are  embedded   non-intersecting
geodesically  convex domains (or `obstacles')
$\ocal_j$.  In the case where $X$
is a flat
torus or a square, such a billiard is often  called a  Lorentz-Sinai dispersing billiard. We refer to  \cite{CM} (see also
\cite{JZ}) for background
and references to the literature. Billiards on  $(M, g)$ of the form \eqref{M}  never have  self-focal points. Self-focal points $q$ are necessarily self-conjugate, i.e. there
exists a broken  Jacobi field along a geodesic billiard loop at $q$ vanishing at both endpoints. But as shown 
in  \cite{JZ}, non-positively curved dispersive billiards
do not have conjugate points.  Theorem \ref{SoZthm} applies to these examples.




\subsection{\label{EXAMPLEINTRO} Polar caps on spheres}   A family of simple examples of Riemannian manifolds with
concave boundary is given by certain complements of polar caps on standard spheres  $S^n$. 
Let   $S^n_r : = S^n \backslash B_r(p)$ be the complement
 of a `polar cap' of radius r.  For $r < \pi/2$, $S^n_r$ is a
concave domain in $S^n$. When $r = \frac{\pi}{2}$, $S^n_r$ becomes an upper hemisphere with totally geodesic (hence
weakly concave)  boundary.

The  hemisphere $S^n_+$  gives an example of a billiard with self-focal points at the boundary. Every geodesic leaving a point
$q \in \partial S^n_+$ returns to $\partial S^n_+$ at the point $-q$  at time $\pi$ , then reflects from the boundary and returns to $q$ at time $2 \pi$.
The reflected geodesic is the mirror image of the second half  of
the corresponding great circle on $S^n$. If $(S^2, g)$ is a convex surface of revolution for which $x \to -x$ is an isometry, then the same is true
for each hemisphere bounded by the equator.

      For the hemisphere, the Cauchy data of   eigenfunctions saturate the remainder bounds and sup norm of Corollary \ref{REMJUMP}. We prove this in Section \ref{EXAMPLESECT} but 
sketch the idea here in the case of Neumann eigenfunctions.
The  Neumann eigenfunctions of the upper hemisphere are restrictions of  even eigenfunctions of $S^n$ under the 
map  $x_n \to - x_n$, while the Dirichlet eigenfunctions
are restrictions of the odd ones.   The extremal Neumann eigenfunctions are restrictions of zonal spherical harmonics with pole on the 
boundary;  zonal means rotationally invariant for the axis through the pole (and its anti-pode).    When $ r < \frac{\pi}{2}$,
so that the boundary is strictly concave, the sup norm bounds on boundary traces are not achieved.  When $r > \frac{\pi}{2}$
the boundary is convex, and again the sup norm bounds on Cauchy data of eigenfunctions are not achieved.  
 All types of extremal behavior (i.e. for all $L^p$ norms)   occur for the hemisphere.   

 Dirichlet or Neumann eigenfunctions on complements of polar caps can be constructed 
using separation of variables. The radial factors are given by Legendre functions that may be singular at the poles. 
An example of an eigenfunction is the Green's function $G(\lambda, x, q)$ with pole $q$ at the north pole of $S^n_r$
with respect to a choice of maximal abelian subgroup of $SO(n + 1)$. The Green's function satisfies
$(\Delta + \lambda^2) G(\lambda, x, q)  = 0$ in $S^n_r$ since the $\delta_q$ lies outside the domain. It is rotationally
invariant and therefore satisfies the Dirichlet boundary condition if $\lambda$ is chosen to that $G(\lambda, r, q) = 0$
and satisfies Neumann boundary conditions if $\partial_r G(\lambda, r, q) = 0$. However, it is never an extremal
eigenfunction.

\subsection{A question}

The  example of polar caps suggests that a necessary condition that the bounds of Corollary \ref{REMJUMP} are saturated is that $\partial M$ is totally geodesic. In other words, 
  we ask if the sup norm bounds are ever achieved if the boundary
is strictly concave. 
Note that  concavity of the boundary implies non-existence of self-focal points on the boundary if the curvature is $\leq 0$.

\subsection{Acknowledgements}

We thank Richard Melrose for comments on the article. We also thank S. Ariturk and  the three referees  for
 helpful  corrections of notational ambiguities and typos as well as for
 discussions of \cite{M}. In particular we thank  J. Galkowski for guiding
 us through the analysis of orders of amplitudes and symbols (see \cite{Gal}
 for the analysis in a closely related problem).

\section{\label{WF} Wave front set bounds on Cauchy data of the wave group   }

The sup norm bounds of Proposition \ref{PTWEYL} and Theorem \ref{SoZthm}   are derived from an analysis of the singularities of the
boundary trace of the wave kernel along the diagonal,
\begin{equation} \label{BTW} E_B^b(t, q, q) =  \sum_{j = 1}^{\infty}  \cos (t \lambda_j)
|\phi_j^b(q)|^2. \end{equation}

In this section we calculate the wave front set of the boundary restriction  (Cauchy data) of the wave group with Dirichlet
or Neumann boundary conditions. The results were also stated and used in \cite{JZ} and \cite{Z4}.

\subsection{Wave front set of the wave group}

We denote by
\begin{equation} \label{EBSB} E_B(t) = \cos \big( t \sqrt{ -
\Delta_{B}}\;\big), \;\;\; \mbox{resp.}\;\;S_B(t) = \frac{\sin
\big(t \sqrt{- \Delta_{B }}\big)}{\sqrt{ -\Delta_{B }}}
\end{equation} the even (resp. odd) wave operators  $(M, g)$ with boundary conditions
$B$. The  wave group  $E_B(t)$  is the
solution operator of the  mixed problem
$$\left\{ \begin{array}{l} \bigl(\frac{\partial^2}{\partial t^2} - \Delta) E_B(t, x, y)=0, \\ \\ \quad E_B(0,x, y)=
\delta_x(y),\;\;\ \frac{\partial}{\partial t}  E_B(0,x, y)= 0,
\;\; x, y \in M;
\\ \\ B E_B(t, x, y) = 0,\;\; x \in \partial M \end{array} \right. $$

The wave front sets of $E_B(t, x, y)$ and $S_B(t, x, y)$ are determined by the    propagation of
singularities theorem of \cite{MSj}  for the mixed Cauchy Dirichlet (or Cauchy Neumann)
problem for the wave equation. We  recall from   \cite{HoI-IV} (Vol.
III, Theorem 23.1.4 and Vol. IV, Proposition 29.3.2) that
\begin{equation} \label{WFEB} WF (E_B(t, x, y)) \subset \bigcup_{\pm} \Lambda_{\pm},  \end{equation} where $\Lambda_{\pm} = \{(t, \tau,x, \xi, y, \eta):\; (x, \xi)= \Phi^t(y, \eta), \;\tau =\pm |\eta|_y\} \subset T^*(\mathbb R \times \Omega \times \Omega)$ is the graph of the generalized (broken) geodesic flow, i.e. the billiard flow $\Phi^t$. The same is true for $WF(S_B)$. As mentioned above, the
broken geodesics in the setting of \eqref{M} are simply the geodesics of the ambient negatively curved surface,
with the equal angle reflections at the boundary; tangential rays simply continue without change at the impact.

\subsection{Restriction of wave kernels to the boundary}

The first tool in the proof of Theorem \ref{SoZthm}   is the analysis of the  restriction of the Schwartz  kernel $E_B(t, x, y) $ of $\cos t \sqrt{-\Delta_B}$
to $\mathbb{R} \times \partial M \times \partial M$ and further to
$\mathbb{R} \times \Delta_{\partial M \times \partial M}$, where $ \Delta_{\partial M \times \partial M}$
is  the diagonal  of $\partial M \times \partial M$.  We denote by $d q$ the surface measure on the boundary
$\partial M$, and by $r u =
u|_{\partial M}$ the trace operator. We denote by $E_B^b(t, q', q) \in \mathcal{D}'( \mathbb{R} \times \partial M
\times \partial M)$ the following boundary traces of the Schwartz kernel $E_B(t, x, y)$ defined in (\ref{EBSB}):

\begin{equation} \label{b}
E_B^b(t, q', q) =
\left\{\begin{array}{ll}
r_{q'} r_{q} \partial_{\nu_{q'}} \partial_{\nu_q} E_D(t, q', q) , & \;\; \;\;\; \mbox{Dirichlet}\\ \\
r_{q'} r_{q}\;  \; E_N(t, q', q), & \;\; \mbox{Neumann}
\end{array} \right. \end{equation}
The subscripts $q', q$ refer to the variable involved in the
differentiating or restricting.

\subsection{Wave front set of the restricted wave kernel}

%

The first and simplest piece of information is the wave front set of \eqref{BTW}. It follows from \eqref{WFEB}, and from
  standard results on pullbacks of wave front sets under maps,  that the  wave front set of  $E_B^b(t, q, q')$ consists of co-directions of broken trajectories which begin and end on $\partial M$.
That is,
\begin{equation} \label{WF1}
\begin{array}{lll}  
 WF ( E^b_B(t, q, q'))
&\subset & \{(t, \tau, q,
\eta, q', \eta') \in T^* \R \times  B^* \partial M \times B^* \partial M: \\&&\\&&[\Phi^t(q, \xi(q, \eta))]^T = (q', \eta'), \;
\tau =  |\xi| \}.
\end{array}
\end{equation}
Here, the superscript $T$ denotes the tangential projection to $B^* \partial M$.
 We refer to Section 2 of \cite{HeZ} for an extensive discussion.
It follows from \eqref{WF1} that  \begin{equation} \label{WF2} \begin{array}{lll} 
 WF (E_B^b(t, q, q'))
& \subset &  \{(t, \tau, q, \eta, q, \eta') \in T^* \R \times  B^*_q \partial M \times B^*_q \partial M: \\&&\\&& [\Phi^t(q, \xi(q,
\eta))]^T = (q, \eta'),\;\; \tau =  |\xi(q, \eta)|\}. \end{array}  \end{equation}
Thus, for $t \not= 0$, the  singularities of the boundary  trace 
$E_B^b(t, q,q)$ at $q \in \partial M$  to broken
bicharacteristic
 loops based at $q$ in
$\overline{M}$. When $t = 0$ all inward pointing co-directions belong to the wave front set.

\begin{remark} One of the principal  features of the boundary trace  $E_B^b(t, q, q)$ along the diagonal is that
the singularity at $t = 0$ becomes uniformly isolated from other
singularities, while the interior kernel $E_B(t, x, x)$ has
singularities at $t = 2 d(x)$ arbitrarily close to $t = 0$. Here, $d(x)$ is the distance from $x$ to $\partial M$. \end{remark}

\section{\label{PARA} Parametrices for the wave group on manifolds  with concave boundary}

To prove Proposition \ref{PTWEYL}  and Theorem \ref{SoZthm}, we need more than wave front set bounds
on the Cauchy data of  Dirichlet or Neumann wave kernels $E_B(t, q, q')$ along the boundary. We will also
need to have explicit control over the singularity  of $E_B(t, q, q)$ at $t = 0$ and $q \in \partial M$, and 
approximate control over certain microlocalizations $a_h(q, D_q) E_B(t, q, q') |_{q = q'}$. To obtain such control,
we use the Melrose-Taylor  parametrix for $E_B(t, x, y)$ for wave kernels on manifolds with concave boundary  \cite{MT,M,Tay1,
Tay2}.
We only need the parametrix along  $\partial M \times \partial M$, and this is a substantial simplification. 
We follow \cite{M} closely in our analysis of the parametrix and the singularity at $t = 0$.

\subsection{\label{KSECT} Kirchhoff formula} In this section we review the classical Kirchhoff formula for the fundamental
solution of the wave equation. It is implicitly used in \cite{M} and there in  Section \ref{MD}. The main point is that it illustrates the crucial role of the Neumann
operator in the construction of the Dirichlet or Neumann wave kernel.  Background on Kirchhoff's formula in varying degrees of generality can be found in \cite{Sob} (Lecture 14), 
Taylor's notes \cite{Tay3} (section 2),  Theorem 4.1.2 of \cite{Fr}, \cite{F} (page 10). The Kirchhoff formulae show that
that the parametrix construction for the wave kernel can be reduced to that for the (Dirichlet-to-) Neumann operator
$\ncal$ and its inverse. This clarifies the relation between the parametrix constructions for Dirichlet vs. Neumann
boundary conditions. It is implicitly used in Section 2 of \cite{M} and we digress to 
explain to give some of the relevant background.

We assume as above that $(M, g) \subset (X, g)$ where $(X, g)$ is a compact $C^{\infty}$ Riemannian manifold
without boundary. We denote the Schwartz kernel
of $\cos t \sqrt{-\Delta_X} $ by $E_X(t, x, y)$. 
The Kirchhoff formula is  a  layer potential formula for the solution of a certain mixed Cauchy and boundary problem
for the wave equation  $\Box u = (\frac{\partial^2}{\partial t^2} - \Delta_g) u = 0$ on $X \times \R$,
\begin{equation} \label{MIXED} \left\{ \begin{array}{l}  \Box u = 0, \; \rm{on}\; M \times \R \;\; \\ \\
u = g \;\; (\mbox{Dirichlet})\;\; \mbox{or} \;\; \partial_{\nu} u = h \; (\mbox{Neumann}) \; \mbox{on} \; \partial M \times \R\\ \\
u = 0, \;\; t \leq 0 \end{array} \right. \end{equation}
where $\partial_{\nu}$ is the outer inner unit normal to $M$.
The  {\it  Kirchhoff formula } is the Green's formula,
$$ u(t, x)  =   \int_0^t \int_{\partial M} (\partial_{\nu_y} G(s, t, x, y) )u(s, y)  - G(s, t, x,  y)
\partial_{\nu_y} u(s, y) d S(y) ds,  $$
where $G$ is the forward fundamental solution of $\Box$ on $X$,
$$
G(t, s, x, y) = H(t - s)  \frac{\sin (t - s) \sqrt{- \Delta_X}}{\sqrt{- \Delta_X}}.
 $$ Here, $H(t) = {\bf 1}_{t \geq 0}$ is the Heaviside step
function. The  complication is that the formula involves both $ u$ and $\partial_{\nu } u$ on $\partial M$,
but only half the data is prescribed by the equation and the other half
requires the use of the (Dirichlet-to-) Neumann operator $\ncal$.\footnote{We follow the terminology in \cite{M} of calling $\ncal$ the Neumann operator.} We recall that, 
given the data $u |_{\partial M \times \R}$, $$\partial_{\nu_y} u(s, y) = \ncal (u |_{\partial M \times \R}).$$

We now apply the Kirchhoff formula to obtain expressions for the Dirichlet, resp. Neumann, wave kernels of $(M, g)$.
They have the form $$E_B(t, x, y) = E_X(t, x, y) + C_B(t, x, y)$$
where $E_X$ is the `free' wave kernel of $(X, g)$ and $C_B$ is the compensating kernel, designed so that the sum
satisfies the boundary condition. Thus, $$\left\{ \begin{array}{l}  \Box C_B(t, x, y)  = 0, \;\; \\ \\
C  = - E_X  \; \mbox{Dirichlet}\;\; \mbox{or} \;\; \partial_{\nu} C = - \partial_{\nu} E_X \; \mbox{Neumann} \; \mbox{on} \; \partial M \times \R\\ \\
C_B = 0, \;\; t \leq 0. \end{array} \right. $$

\begin{lemma} In the Neumann case, 
$$\begin{array}{lll} E_N(t, x, y)  & = & E_X(t, x, y) +    \int_0^t \int_{\partial M} (\partial_{\nu_y} G(s, t, x, q') ) \ncal^{-1} \partial_{\nu_y} E_X(s, q', y)  \\ &&\\ &  & - G(s, t, x,  q')
\partial_{\nu_y}  E_X (s, q', y)) d S(q') ds. \end{array}$$

In the Dirichlet case,  

$$\begin{array}{lll} E_D(t, x, y) &  = & E_X(t, x, y) +   \int_0^t \int_{\partial M} (\partial_{\nu_y} G(s, t, x, q') )  E_X(s, q', y)  
 \\ &&\\ &  &  - G(s, t, x,  q')
\ncal  E_X(s, q', y) d S(q') )ds. \end{array} $$

\end{lemma}
As the Lemma shows, the main difference between the Neumann and Dirichlet cases is that one needs to
construct $\ncal^{-1}$ in the Neumann case and $\ncal$ in the Dirichlet case. See \cite{Tay2} (section X.5), \cite{F} or pages 262- 263 of \cite{M} for
more on this representation and on the Neumann operator and its inverse. Since there is a parametrix for $G(s, t, x, y)$ the main problem is to construct a parametrix for $\ncal$. We do not
review more of the theory but head straight for Melrose's analysis of the parametrix in the case of concave boundary.

\subsection{\label{MD} Microlocal decomposition}

 As on
the bottom of page 261 of \cite{M},  we introduce Fermi normal (geodesic) coordinates  $z = (x,y)$ near $\partial M$,
where $x$ is the normal variable, giving a   defining function of $\partial M$ and $y$ gives coordinates on
$\partial M$.
Then  $\Delta = \frac{\partial^2}{\partial x^2}  + Q(x, y, \partial_y)$, where
the symbol $q(x, y, \eta)$ is the induced Riemannian norm on the level
sets of $x$.

We recall (see \cite{M,MT, HoI-IV,F})  that there exist two regions of $T^* \R \backslash \{0\} \times T^* \partial M$: 
\bigskip

 \begin{itemize}

\item the hyperbolic region(s) $H_{\pm}$ where $\tau^2 > q(0, y, \eta)$, 
\bigskip

\item the elliptic region $E$ where $\tau^2 < q(0, y, \eta)$, 
\bigskip

\item which are separated by the  glancing hypersurfaces $G_{\pm}$ where $\tau^2 = q(0, y, \eta)$, $\pm \tau > 0$. 

\end{itemize}
\bigskip

As discussed on  \cite[page 265]{M}, a   solution to the initial-boundary  value problem
$$\left\{ \begin{array}{l} \Box w = 0, \\ \\
\partial_{\nu} w |_{\partial M} = 0, \\ \\
w(0, z) = w_0(z), \; \partial_t w (0, z)  = 0, \end{array} \right. $$
can be obtained by first solving the initial value problem and  then adding compensating term  so that the boundary condition is satisfied. We described this  method
in Section \ref{KSECT}.

A solution to the initial value problem 
can be decomposed as \begin{equation} \label{udef} u = u_T^{\pm} + u_G \end{equation} (transversal plus glancing terms) where $u_G = u_G(t, x, y)$ has the form
\begin{equation} \label{uG} u_G = \int e^{i \Phi(x, t, y, s, \tau, \eta)  - i \Phi(x', 0, y', r, \tau, \eta)} c_G(x, t, y, s, r, \tau, \eta) d \tau d \eta
dr ds, \end{equation}
where the phase has the form,
\begin{equation} \label{PHI} \Phi(x, t, y, s, \tau, \eta) = \alpha(x, t, y, \tau, \eta) + s |\eta|^{1/3} \beta(x, y, \tau, \eta) + \frac{1}{3} s^3 |\eta|, \end{equation}
and where  $c_G$ is a first order classical symbol suppored in $|s|, |r| \leq \epsilon,  |\tau^2 - |\eta|^2| \leq \epsilon \tau^2$ for some $\epsilon > 0$.
The transversal term $u_T^{\pm}$ is a standard Fourier integral operator.

It is shown in  \cite[p. 266-267]{M}  that the  Dirichlet or Neumann cosine wave kernel $w(t, q, q')$  can be
constructed in this way as  a sum of terms
$$w = w_1 + w_2^+ + w_2^- $$
with 
\begin{equation} \label{MICRODECOMP} \left\{ \begin{array}{l} w_1 = (u_G)_c + w_{1,G} + w^+_{1, H} + w^-_{1, H} + w_{1, E} \\ \\
w_2^{\pm} = (u_H)_c + w^{\pm}_{2,G} + w^{\pm, +}_{2, H} + w^{\pm, -}_{2, H} + w^{\pm}_{2, E}. \end{array} \right.\end{equation}
\bigskip

Here $u_c $ is the cutoff of  the kernel $H(t) u(x, t, y; x', y')$ to $x < 0$ where $x$ is the normal coordinate
to $\partial M$.  
The hyperbolic and elliptic terms (with subscripts H, E)  are of a standard kind (although the elliptic term is a Fourier integral operator with
complex phase). In generalizing the arguments of \cite{SZ1,SZ2}  to the boundary case, these terms do not require
any essential modification and we therefore  suppress them in the exposition.

As mentioned above, the main difference between the Dirichlet and Neumann
cases is that one needs to construct the Neumann operator $\ncal$  as a Fourier-Airy integral operator in the Dirichlet case and its inverse $\ncal^{-1}$ in the Neumann case. 

In the following we need some notation pertaining to the Airy function
$Ai(s)$. Following \cite{MT} we  set 
\begin{equation} \label{Apm} A_{\pm}(s) = Ai(e^{\pm \frac{2 \pi i}{3} } s). 
\end{equation}
See   Section \ref{AIRYAPP} for its properties.

\subsection{\label{GTSECT} Glancing terms}

The novel feature of the concave boundary case is the analysis of the glancing terms $w_{1, G}, w_{2, G}^{\pm}$. 
In \cite[(3.1)]{M}, a parametrix for the glancing part is given in the form,
with $z = (x, y), z'' = (x'', y'') \in M$,

\begin{equation} \label{w1G21} \begin{array}{lll} w_{1, G}(t, z, z'')  &=& \int \int_0^{\infty}   e^{i \alpha(x, t, y, \tau', \eta') - i t' (\tau' - \tau) \mu 
+ i s \beta_0' |\eta'|^{\frac{1}{3}} + i \frac{s^3}{3} |\eta'|  } \\&&\\&& e^{- i \alpha(x'', 0, y'', \tau, \eta')  - ir \beta''
|\eta'|^{\frac{1}{3}} - i \frac{r^3}{3} |\eta'| }
\frac{a  A'_{\pm}(\beta) + b A_{\pm}(\beta) }{A_{\pm}(\beta_0) }\\ &&\\&& 
 dt'  d s dr d \tau d \tau'  d \eta'. \end{array} \end{equation}
 Here, the half-line integral $\int_0^{\infty}$ refers to $dt'$ and $A_{\pm}$ is
 defined in \eqref{Apm}. Note that there are two $w_{1,G}$ corresponding
 to the choice of $\pm$ but as in \cite{M} we suppress this in the notation
 since the two cases are similar. In Section \ref{DER} we simplify the integral
 and restrict it to the boundary.

The  $s, r$ variable  arose  in \cite[(2.16)-(2.17)]{M} in the phase \eqref{PHI}
defined near $\tau = \pm |\eta|$. 
Here,   $x = 0$ defines $\partial M$ and the rest of the notation is defined as follows:
\bigskip

\begin{enumerate}
\item The `phases'  $\alpha, \beta$ are real $C^{\infty}$ functions, homogeneous of degree $1$ resp. $\frac{2}{3}$ in $(\tau, \eta)$.
Moreover,  $\alpha(x, t, y, \tau, \eta) $ is linear in $t$, invariant under $(t, \tau) \to (-t, - \tau)$ and
$\det (\partial_{t, y} \partial_{\tau, \eta} \alpha(0, t, y, \tau, \eta) \not= 0$ on $\tau^2 = |\eta|^2$ (\cite{M} (2.8), (2.10)).
Hence 
\begin{equation}\begin{array}{l}  \label{alpha} - \alpha(0, t', y', \tau', \eta') + \alpha(0, t', y', \tau, \eta) \\ \\= \langle y' + t' g, \eta - \eta'\rangle + t'(\tau - \tau') \mu \end{array} \end{equation}
where $g, \mu$ are $C^{\infty}$ and homogeneous of degree zero with $\mu \tau > 0$ and $\langle \cdot, \cdot \rangle$ is the inner product.
\bigskip

\item $\beta(x,y,  \tau, \eta)$ is independent of $t$ and along the boundary is given by 
\begin{equation} \label{b0} \beta_0: = \beta(0, y, \tau, \eta) = (\tau^2 - |\eta|^2) |\eta|^{-\frac{4}{3}},\;  \end{equation}
see (\cite[(2.9) and above (3.2)]{M}. Note that $\beta$ is homogeneous of degree $2/3$, and that $\beta_0 > 0$ in the hyperbolic set.\footnote{The 
factor $ |\eta|^{-\frac{4}{3}}$ is missing in \cite[(2.9)]{M} but corrected in 
\cite[(3.2)]{M}. }

  When evaluating $\beta, \beta_0$ at  primed coordinates, we abbreviate the value of $\beta$ by priming it as follows:
\begin{equation} \label{beta} \beta_0' = (\tau^2 - |\eta'|^2)  |\eta'|^{-\frac{4}{3}}, \;\;\;
\beta'' = \beta(x'', 0, y'', \tau, \eta'). \end{equation}  \bigskip

\item $\Phi(0, t, y, s, \tau, \eta) = \alpha(0, t, y,  \tau, \eta) + s |\eta|^{\frac{1}{3}} \beta (0,  y,  \tau, \eta)
+ \frac{1}{3} s^3 |\eta|. $ \cite[(2.16)]{M}.\bigskip

\item $a, b$ are not classical symbols but are supported near $s = r = 0, \tau^2 = |\eta|^2$ and have the following  form \cite[(2.14)]{M}.
\begin{equation} \label{2.14} \begin{array}{lll}  \tilde{a}_G & \sim & \sum_{j \geq 0} a_{j + 1}(x, t, y, t', y', \tau, \eta) \Phi_{\pm}^{-1, j}(\beta(0, \tau,\eta)) \\&&\\&& + a_0(x, t, y, t', y', \tau, \eta) \end{array}\end{equation}
(with different coefficients for $a$ resp. $b$) where $a_j, b_j$ are classical symbols and
$\Phi_{\pm}^{-1}(z) = \frac{A_{\pm}(z)}{A_{\pm}'(z)}, $ and with 
$\Phi_{\pm}^{-1, j}(z) = \frac{d^j}{dz^j} \Phi_{\pm}(z).$  See Section \ref{ORDERsect}
for a discussion of the orders of the terms.

\bigskip

\item Recalling \eqref{udef},  $u_G$ is given in  \cite[(2.17)]{M} and $u_c$ is obtained by cutting off $u' := H(t) u(x,t,y; x', y') $ in $x\leq 0$.

\end{enumerate}
\bigskip

The corresponding term(s) for $w_2^{\pm}$ have a similar form with $\Phi$ replaced by $\phi_{\pm}(0, z', \zeta)
= z' \cdot \zeta$ (\cite[(2.1)]{M}).  The cutoff term $(u_G)_c$ is also handled in the same way. Hence we focus
on $w_{1, G}$.

\subsection{\label{DER} Simplification of \eqref{w1G21}}

 Following  \cite[p. 268-269]{M} and starting from \eqref{w1G21}  (\cite[(3.1)]{M}), one simplifies the integral by 
using statonary phase to 
eliminate the $(t', \tau)$ integral  (changing the order of the amplitude by $-1$) to get the
forms  $w_{1,G} = w^{(1)}_{1,G} + w^{(2)}_{1,G}$ of \cite[(3.2)-(3.3)-(3.4)]{M}.

After restriction of $w_{1,G}$ \eqref{w1G21}  to points $(q, q'') \in \partial M \times \partial M$, we have
\begin{equation} \label{w1} \begin{array}{lll} w_{1, G}(t, q, q'')  &=& \int \int_0^{\infty}   e^{i \alpha(0, t, q, \tau', \eta') - i t' (\tau' - \tau) \mu 
+ i s \beta_0' |\eta'|^{\frac{1}{3}} + i \frac{s^3}{3} |\eta'|  } \\&&\\&& e^{- i \alpha(0, 0, q'', \tau, \eta')  - ir \beta''
|\eta'|^{\frac{1}{3}} - i \frac{r^3}{3} |\eta'| } \frac{a  A'_{\pm}(\beta_0) + b A_{\pm}(\beta_0) }{A_{\pm}(\beta_0) }\\ &&\\&& 
  d s dt' d \tau dr  d \tau'  d \eta'. \end{array} \end{equation}
 As above, the $dt'$ integral is only over $\R_+$.
Following  \cite[p. 268-269]{M} and starting from \eqref{w1G21}  (\cite[(3.1)]{M}), one simplifies the integral by 
using statonary phase to 
eliminate the $(t', \tau)$ integral  (changing the order of the amplitude by $-1$).  

 In fact, we are only
 interested in the diagonal of the kernel, which is given by
\footnote{There is a typo in the second $i$  in \cite[(3.3)]{M} in $is^3$,
$ir^2$ should be $ir^3$ and $ir \bar{\beta}$ should be $-i r \bar{\beta} |\eta'|^{1/3}$. }  
\begin{equation} \label{w1G} \begin{array}{lll} w^{(1)}_{1,G}(t, q, q)  & = & \int_{\tau' (s - r) \leq 0} e^{i \alpha(0, t, q, \tau', \eta')
- i \alpha(0, 0, q, \tau', \eta') + is \beta_0 |\eta'|^{\frac{1}{3}} + is^3 \frac{|\eta'|}{3} - i r \beta_0 |\eta'|^{1/3} - ir^3 |\eta'|/3} \\ &&\\
& &
\frac{a  A_{\pm}(\beta_0) + b A'_{\pm}(\beta_0) }{A_{\pm}(\beta_0) }  \;\; ds dr \;  d\tau' d \eta'. \end{array} \end{equation}
\subsection{\label{ORDERsect} Homogeneities of amplitudes and symbols}
In this section, we review the orders of the various amplitudes
appearing in the glancing term \eqref{w1G21} or the simplified form
on the boundary \eqref{w1G}.  For the background on Airy functions
and Airy quotients we refer to the Appendix (Section \ref{AIRYAPP}), which 
is essentially from \cite{MT} Appendix A.3. We denote symbols of order
 $m$ in some Hormander symbol space $S^m_{\rho, \delta}$ simply
 as belonging to $S^m$.
 
 In calculating orders of amplitudes we use only two facts: (i) the symbolic
 properties of Airy quotients $\Phi_{\pm}$ in Section \ref{AIRYAPP}; and
 (ii) the fact that the glancing parametrices $w_{1,G}, w^{\pm}_{2,G}$
 are (for fixed $t$) Fourier-Airy integral operators of order $0$, since they are parts of
 $\cos t \sqrt{-\Delta}$ in which one microlocalizes to an $\epsilon$ neighborhood of the glancing set. The microsupport of the  parametrix includes
 a sector of hyperbolic points and so the amplitudes must have the same
 order as in the hyperbolic terms of the parametrix.

\begin{rem} \label{WARN}The orders of the amplitudes stated at the end  of this section refer to \eqref{w1G21} or equivalently
to \cite[(3.1)]{M}. But the order analysis in \cite{M} pertains to \cite[(2.22)]{M}. So we track
the change in orders from \cite[(2.22)]{M} to \cite[(3.1)]{M} in this section.

Below,  in the proof of Lemma \ref{MLEM},  further  modifications are made to simplify \eqref{w1G21} or  \cite[(3.1)]{M}. The
new amplitudes are still denoted by $a,b$ (as in \cite{M}) but their
orders change in the modifications. The orders of the amplitudes of \eqref{w1G21} are stated here so that one can keep track of how the orders change in the course of the modifications.

We also warn that the roles of the coefficients $a,b$ in \cite{M} get reversed
in the course of the proof. \footnote{ $a$ is the coefficient of the primed term in   \cite[(2.22)]{M} for $w_{1,G}$ and \cite[(2.26)]{M} for $w^{\pm}_{2,G}$. 
Warning: the roles of $\tilde{a}_G$ and $\tilde{b}_G$ were switched in\cite[(2.22)]{M}.
In \cite[(2.14)]{M}, $\tilde{a}_G, \tilde{b}_G$ have roles consistent with
Remark \ref{WARN}. 
The roles are reversed below \cite[(3.1)]{M}.}
 We will always use $a$ for the coefficient
of $ \frac{A_{\pm}(\beta)}{A_{\pm}(\beta_0)}$ and $b$ for the coefficient
of $ \frac{A'_{\pm}(\beta)}{A_{\pm}(\beta_0)}$.
\end{rem}

  For the representation \cite[(3.1)]{M} it is said on  \cite[p. 268]{M},  that the amplitude has the form,
 \begin{equation} \label{Ass} As_{\pm}(\beta, \beta_0) = a \frac{A_{\pm}(\beta)}{A_{\pm}(\beta_0)} + b
 \frac{A_{\pm}'(\beta)}{A_{\pm}(\beta_0)}. \end{equation}
It is explained in \cite[(2.14)]{M} that the coefficients $a, b$ have the form \eqref{2.14}.
More precisely, one has symbol expansions,
\footnote{The orders of the amplitudes are apparently stated incorrectly in \cite
[(2.13)-(2.14)]{M} 
 }
 \begin{equation} \label{aGbG} \left\{ \begin{array}{l} \widetilde{a}_G \sim \sum_{j \geq 0} 
 a_{j+1}  \Phi_{\pm}^{-1, j}(\beta_0(\tau, \eta)) + a_0,\\ \\
  \widetilde{b}_G \sim \sum_{j \geq 0} 
 b_{j+1}  \Phi_{\pm}^{-1, j}(\beta_0(\tau, \eta)) + b_0, \end{array} \right. \end{equation}
 Here,
 $\Phi_{\pm}^{-1}(z) = \frac{A_{\pm}(z)}{A'_{\pm}(z)}$ and  $\Phi^{-1, j}$ is the jth derivative. As reviewed in Section \ref{AIRYAPP}, $\Phi^{-1}(\beta_0)
 \in S^{-\frac{1}{3}}, \Phi^{-1, j}(\beta_0) \in S^{-\frac{1}{3} - j}.$   
Indeed, the   Airy quotients $\Phi_{\pm}$ \eqref{Phipm} are classical symbols of order $\half$,
$$\Phi_{\pm}(\zeta) \sim \sum_{j \geq 0} a_j^{\pm} \zeta^{\half - \frac{3 j}{2}}, \;\; \Re \zeta \to \infty, $$
and $\Phi_{\pm}(\beta) \in S^{\frac{1}{3}}_{\frac{1}{3}, 0},$  since $\beta$ is homogenous of order $ \frac{2}{3}.$ 
For purposes of this article, only  the boundary restriction is relevant and then \eqref{Ass} becomes  \begin{equation} \label{ab1} As_{\pm} (\beta_0, \beta_0) = a \frac{A_{\pm}(\beta_0)}{A_{\pm}(\beta_0)} + b
 \frac{A_{\pm}'(\beta_0)}{A_{\pm}(\beta_0)} = a + b  \frac{A_{\pm}'(\beta_0)}{A_{\pm}(\beta_0)} ,\end{equation}
 where $a, b$ are classical symbols.   
  In the final form \eqref{ab2}, after various modifications,   $a \in S^1, \; b \in S^{2/3}$ and both terms have order 1.
  However, in this section we are discussing  the representations \eqref{w1G21} and \eqref{w1G}.
  
  \begin{lemma} \label{ORDERLEM} In the representation \eqref{w1G}, 
  $\frac{a  A'_{\pm}(\beta) + b A_{\pm}(\beta_0)}{A_{\pm}(\beta)}$ and \eqref{ab1}  are symbols of order 1. Hence $a$ has order 1 and $b$ has order $\frac{2}{3}.$

\end{lemma}


\begin{proof} Granted that $a,b$ have symbol expansions of the form \eqref{aGbG} and that $w_{1,G} $
is a Fourier-Airy operator of order $0$ we can read off the orders  from the integral representations
\eqref{w1G} (see  \cite[(3.1)-(3.2)-(3.3)]{M}).

The full amplitude is $\rm{As}(\beta, \beta_0)$ of \eqref{w1G21} before the boundary restriction  is with  respect to $dt' d\tau' d\eta' d \tau ds dr.$ After the $dt' d\tau$ integral it is with respect to $dr ds d \tau' d \eta'$. As in the hyperbolic
term the amplitude must have order $1$ to obtain an oscillatory integral representation for a Fourier-Airy
integral operator of order $0$. (In the more familiar expressions with only the phase variables
$d \tau' d\eta'$ of $T_q^* M$ the amplitude would have order zero. The order $1$ compensates for
the additional $dr ds$ integral).

\end{proof}

\begin{rem}
 
 In \cite[(2.14)]{M} it is said that   $a_j, b_j$ are classical symbols of orders $- \frac{1}{3} - \frac{j}{3},
- \frac{2}{3} - \frac{j}{3},$ respectively. The corresponding $\tilde{a}_G,
\tilde{b}_G$ refer to the oscillatory integral representation \cite[(2.13)]{M} for $v_G$ \cite[(2.5),(2.13)]{M} with phase variables $dt' d\tau dy' dy$. This
is used to construct the glancing parametrix $w_{1,G}$ \cite[(2.22)]{M}, which has
an amplitude of the form \eqref{Ass} with $a = b_G, b = a_G$ times
another classical symbol $d_G$ of order 2. In this expression, $a_G$ has  order -4/3 and $b_G$ has order $ -1$ and have expansions of type
\eqref{aGbG} (see  below \cite[(2.24)]{M}. There are other terms $w_{2, H}^{\pm}$ of
a similar form given in \cite[(2.26)]{M}. In the next expression \cite[(3.1)]{M}
for $w_{1,G}$, the factor $d_G$ is absorbed into the coefficients $a,b$.
\end{rem}

This finishes our review on the basic objects and notations in \cite{M}.

\subsection{Comparison of \cite{M} to other articles on diffraction.}

 Since the notation $\alpha, \beta$ for the phases does not seem to appear in any other article on diffraction problems, we pause to relate the notation  to that
 appearing in \cite{Fr76,FM77,M75,Tay2}. In \cite{M75}, the phases $(\alpha, \beta)$ of \cite{M} are denoted by $(\phi, \zeta)$. They are described rather
 explicitly in Proposition 6.3 of \cite{M75}. The same coupled eikonal equations
 for the two phases  had
 been analysed earlier by D. Ludwig \cite{L67}. The phase $\beta$ of \cite{M}
 is modeled after the phase of the Friedlander model  \cite{Fr76, FM77}. In those articles,  the  coordinates are denoted  $(x, y_n, y')$ corresponding to
 $(x, t, y)$ in \cite{M}. Let $(\xi, \eta_n, \eta')$ denote the symplectic dual
 coordinates, so that $\eta_n $ corresponds to $\tau$ and $\eta' $ to $\eta$ in \cite{M}. Then the phase $\zeta_0$  in  \cite[(2.2)]{FM77}  corresponding to $\beta_0$ in \cite{M} is (cf. \cite[(3.5)]{Fr})
 $$\zeta_0 (\eta) = \eta_n^{-4/3} (|\eta'|^2 - \eta_n^2). $$
 As in \eqref{b0} it is independent of $y_n (= t)$ and $y'$.
 The full phase of the Friedlander model corresponding to $\beta(x, y, \tau, \eta)$ in \cite{M} is
 $$\zeta(x, \eta_n, \eta') = \eta_n^{-4/3} (|\eta'|^2 -  (1 + x) \eta_n^2). $$
 Of course, the phase $\beta$ is much more complicated for $x >0$
 but is related to $\zeta$ by the canonical transformation to Friedlander
 normal form \cite{M76}, which is implicitly used in \cite{M}.
 
 The angular phase $\alpha$  corresponds to the phase   $\theta = t |\xi| + \psi(x, \xi, \eta)$ of \cite{Tay2}. In the case of the exterior of the unit ball in 
 $\R^n$, $\theta$ is closely related to the usual polar coordinates (see \cite{L67}).

The mixed Cauchy-boundary problem  for the wave group \eqref{MIXED}
 is rarely studied explicitly in  the literature on diffractive parametrices. To our knowledge, the first direct treatment of parametrices for  the propagator
$\exp i t \sqrt{-\Delta}$ for the exterior
of a convex obstacle (in $\R^n)$ is given by Farris in \cite{F}. The foundational articles \cite{M75,Tay1} do not directly address the propagator
(nor the related solution operators $\cos t \sqrt{-\Delta}$ or $\frac{\sin t \sqrt{-\Delta}}{\sqrt{-\Delta}}.)$ Farris used the techniques of \cite{Tay1} to analyze the propagator, but only for  Dirichlet boundary conditions. 
Neumann boundary conditions are more difficult because one has to 
invert the Dirichlet to Neumann operator $\ncal$. As far as we know, 
Melrose's article \cite{M} is the first to write down explicit expressions for
the diffractive parametrices for the mixed problem with Neumann boundary conditions. It does not seem to refer to any prior works with details on the 
parametrix construction in this case. The monograph \cite{MT} of Melrose-Taylor discussed the parametrix construction using the work of Farris for
Dirichlet boundary conditions and has a chapter on how to modify it for
Neumann boundary conditions;  the details on the latter are left to the reader.
We have not found subsequent articles which give an independent analysis
of the parametrix constructions. J. Galkowski has recently studied a related
diffractive parametrix problem and has corrected some of the statements
in \cite{M} on orders of amplitudes \cite{Gal}.

\section{\label{GLSING} Singularity at $t = 0$ of the Cauchy data of the glancing parametrix: Proof of Lemma \ref{MLEMintro} }

We now begin the proof of Proposition \ref{MELROSE} and the dual (and more precise) statements in  Propositions \ref{PTWEYL}  and \eqref{OpaPTWEYL}. As mentioned above, there is a minimal `loop of length'  $\ell$ of geodesic billiard loops in $M$ which begin and end on
$\partial M$. For $|t| < \ell$ the only singularity of $S_q(t)$ occurs at $ t = 0$. We use the microlocal decomposition
\eqref{MICRODECOMP}
into hyperbolic, elliptic and glancing sets.  
The normality of the singularity at $t = 0$ is standard for the elliptic and hyperbolic terms, 
so we only discuss the glancing terms
$(u_G)_c, w_{1, G}, w^{\pm}_{2, G}$.  Away from the glancing (tangential) directions,
$E^b_B(t, q, q)$ is a Fourier integral kernel whose symbol is computed in \cite{HeZ}. In this section, we consider the contribution of the glancing part
of the wave group to the singularity at $t = 0$ and prove Lemma \ref{MLEMintro}.


We follow the proof of the Weyl law for concave boundary in \cite{M}, but take boundary trace along
the diagonal instead of integrating $E_N(t, x, x)$ over the domain. The proof
is a relatively small modification of the proof of the Weyl law in \cite{M}, and we only explain the modifications and
not the entire proof.  In fact, the   proof of the result for the restriction
to the boundary is simpler.
However, we go into detail on orders of amplitudes
and on the final step using almost analytic extensions to prove symbolic
properties because they are omitted in  \cite{M} and are not always correctly
stated.

\subsection{Proof of Lemma \ref{MLEMintro} for Neumann boundary conditions}

We first consider the Neumann case, which is more difficult than the Dirichlet case because it is necessary to invert the Neumann operator $\ncal$.

 The key point is the following special case  of Lemma \ref{MLEMintro} where $w_G = w_{G,1}$. It is
 the most important term of the glancing parametrix. In Section \ref{ALLwG} we take into account
 the other glancing terms.   Together with standard facts on the 
 hyperbolic and elliptic terms, it  implies Proposition \ref{MELROSE}.

\begin{lemma} \label{MLEM} On the diagonal of $\partial M \times \partial M$,
\begin{equation} \label{edef2} w_{1, G}(t, q, q) = \int e^{i t \theta} e(t, \theta, q) d \theta, \end{equation}
where  $e(t, \theta, q)$  is a classical symbol in $\theta$ of order $ n-1 $ when $ q \in \partial M$.

Moreover, the  contribution of this glancing term
to the first two terms of the expansions of Proposition  \ref{SF2} are of order $O(\epsilon)$ as $\epsilon \to 0$.

 \end{lemma}

We recall that a classical symbol of order $\gamma$ is a polyhomogeneous
function possessing an  expansion,
$$e(t, q, \theta) \sim \sum_{j = 0}^{M} e_j(t, q) |\theta|^{\gamma -j} + R_M(t, q, \theta), \;(
D_{\theta}^m R_M(t, q ,\theta) \leq C(t, x) |\theta|^{\gamma -M-m}). $$
To prove Lemma \ref{MLEM}, we  use the analysis of $e(t)$ from   \cite[(3.5)]{M} but do not integrate in $z$. We first analyze
the difficult  term $w_{1,G}$.

We use the simplified  expressions \eqref{w1G} (see  \cite[ (3.2)-(3.3)-(3.4)]{M}) for  the boundary
values of $w_{1,G}$ and briefly
recall their derivation. It suffices for our purposes to restrict the kernels
to $\partial M \times \partial M$ and this simplifies their form considerably. 
 
At the boundary $x'' = 0$, $\beta = \beta_0$ ($= \bar{\beta} $ in 
  \cite[p. 268-69]{M}) and (cf. \eqref{ab1})
\begin{equation} \label{AS}
\frac{a  A_{\pm}(\beta_0) + b A'_{\pm}(\beta_0) }{A_{\pm}(\beta_0) } = a + b \frac{A_{\pm}'}{A_{\pm}} (\beta_0). \end{equation}
Here, $a,b$ are not classical symbols but have the form  \cite[(2.14)]{M}.  The  amplitudes $a, b$ may be taken to
be  supported
in $|s|, |r| \leq \epsilon, |\tau^2 - |\eta|^2| \leq \epsilon \tau^2$. We also recall
from Section \ref{GTSECT} (see \eqref{alpha}) that  $\alpha$ is linear in t.

The next step is to      change variables\footnote{ $|\eta'|^{-1/3} $ should be $|\eta'|^{1/3}$ and $|\theta|^{1/3}$ should be $|\theta|^{-1/3}$ in the change of variables in \cite[(3.5)]{M}.} in \eqref{w1G} to $$ \lambda = \beta_0 |\theta|^{-2/3}, S = s |\eta'|^{1/3} |\theta|^{-1/3},
R = r |\eta'|^{1/3} |\theta|^{-1/3}. $$ Here,
\begin{equation}\label{thetadef} \theta = \partial_t \alpha, \;\; \rm{so\; that}\; \alpha(0, t, q, \tau', \eta')
- \alpha(0, 0, q, \tau', \eta') = t \theta, \end{equation} 
explaining the $e^{i t \theta}$ factor in \eqref{edef}. We write out the change of variables
since there are several typographical errors above \cite[(3.5)]{M} and because a factor of $\theta^{n-1}$ was omitted. \footnote{Also, $s$
should be $S$ and the domain of integration should be $\theta (S - R(1 + x \gamma)) \leq 0$ in \cite[(3.6)]{M}.} Namely, the change of variables implicitly
involves introduction of polar coordinates in $\R_{\tau'} \times \R_{\eta'}$. 
The polar variable is $\theta$. On the unit sphere $S^{n-1}$, $\lambda$ is a well-defined coordinate and we fill it out to get local coordinates $(\lambda, \omega')$. The coordinates $(\lambda, \omega')$ depend only on $(\tau', \eta')$ and we denote the Jacobian by
$$d \tau' d \eta' = |\theta|^{n-1}J_1(\lambda, \omega')  d \theta d\lambda d \omega'. $$  The change of variables is given by
$$(r, s, \tau', \eta') \in \R_r \times \R_s \times T^*(\R_{t'}) \times T^*_q (\partial M) \to (R, S, \theta, \lambda, \omega') \in \R_{R,S} \times \R_+ \times S^{n-1}. $$
The level sets of $\lambda$ are those of $\beta_0$, so that $\beta_0 = 0$ defines the glancing directions. The full Jacobian is 
$$dr ds d \tau' d \eta' =| \theta|^{n-1}J(\lambda, \omega')  d \theta d\lambda d \omega', $$ 
where $J$ is homogeneous  of degree $0$.

From \eqref{thetadef} and from
$$\left\{ \begin{array}{l} 

|\theta| S \lambda  =  s \beta_0  |\eta'|^{1/3}, \;\;
|\theta| S^3/3  = s^3 |\eta'|/3, \\ \\
|\theta| R \lambda  = r \beta_0 |\eta'|^{1/3}, \;\;
|\theta| R^3/3  = r^3 |\eta'|/3.

\end{array} \right. $$ we find that the phase changes to the following:
\begin{equation} \begin{array}{l}
\alpha(0, t, q, \tau', \eta')
- \alpha(0, 0, q, \tau', \eta') + s \beta_0 |\eta'|^{\frac{1}{3}} + s^3 \frac{|\eta'|}{3} - r \beta_0 |\eta'|^{1/3} - r^3 |\eta'|/3 \\ \\ = t \theta +  \left(S \lambda +  S^3/3 - R \lambda - R^3/3 \right) |\theta|. \end{array} \end{equation}
The phase and Airy quotients  are independent of the $\omega'$ variables and we integrate out
in $J d \omega'$. The change of variables and integration have changed
the amplitudes $a,b$ in \eqref{AS} but for simplicity of notation we still
denote them by $a,b$. The integral is now over $dR dS d \theta d \lambda$.

 This gives \eqref{edef}  where $e(t,q, \theta)$ is  the analogue of \cite[(3.6)]{M},
\begin{equation}\begin{array}{ll} 
e(t, q, \theta) & = |\theta|^{n-1} \int \int_{\theta(S - R) \leq 0}  \int_{\R} e^{ i \left(S \lambda +  S^3/3 - R \lambda - R^3/3 \right) |\theta| }\\&\\ & As_{\pm} (\lambda |\theta|^{2/3},
\lambda |\theta|^{2/3}) dR dS d \lambda, \end{array}\end{equation}
where (cf. \ref{ab1})
\begin{equation} \label{ab2} As_{\pm} =  (a + b \frac{A_{\pm}'}{A_{\pm}} (\beta_0)) \in S^1 \;\; (b \in S^{2/3}, \; a \in S^1), \end{equation} and
 the sign of $\lambda$ is that of
$\beta_0$. 
The Airy amplitude  is a symbol of order $\frac{1}{3} $ of the form,
\begin{equation} \label{Asdef} As_{\pm} (\lambda |\theta|^{2/3},
\lambda |\theta|^{2/3})  \sim \sum_{j = 0}^{\infty} c_j(\lambda, 0) |\theta|^{\frac{1}{3} - j}, \;\;\; |\lambda \theta^{2/3}| \geq 1. \end{equation}
 As mentioned above, in the change
of variables and in the integration over $\omega'$, the coefficients $a, b$ of \eqref{AS} have been changed. 
We continue to denote the new amplitudes by $a,b$.  We refer to Section \ref{AIRYAPP} for further  details. \footnote{We warn that the factor $\theta^{n-1}$ was not put in
explicitly in \cite{M} and the book-keeping of the change in order of $a,b$ was lacking. Also, by comparision with \eqref{ab1} the roles of $a,b$ are interchanged.}

This is the essential starting point for the proof of Lemma \ref{MLEM}. 
The goal is to show that $e(t, q, \theta)$ is a classical symbol  in $\theta$ of order $n-1$, or equivalently that the integral defines a symbol of order $0$. Here, $\theta \in \R$ and there are two sides
accordingly as $\theta > 0, \theta < 0$. We do not have a $dx dy$ integral
as in \cite[(3.6)]{M} because we fix $x = 0$ and $y = q \in \partial M$. The 
amplitudes $a, b$ are supported in a region where both  $S, R$ are of order
 $O(\epsilon)$ where $\epsilon$ is the angle to the tangent plane.  See \cite[(3.9)]{M} in which  the  amplitude
is denoted by $a$.



To prove that $e(t, q, \theta)$ is a symbol of order $n-1$ in $\theta$, we  change variables to $U = S - R, V = S + R$. Then the constraint $\theta (S- R) \leq 0$ becomes
$\theta U \leq 0$, and

\begin{equation}\label{etqtheta}\begin{array}{lll}
e(t, q, \theta) & \sim & |\theta|^{n-1} \int_{\R} \int_{\R} \int_{U \theta \leq 0} e^{ i \left( U \lambda + U Q_2(U, V) ) \right) |\theta| } \\&&\\ && a(U, V, \lambda, \theta)   d \lambda  dU dV,
\end{array} \end{equation}
where $a(U, V, \lambda, \theta)$ is a symbol of order 1 in $\theta$ with compact support in $(\lambda, U, V)$ near $(0,0,0)$. 
Writing the phase as 
$
 (S - R) \lambda + \frac{1}{3} (S - R) (S^2 + S R + R^2), $
and changing variables, the new phase is
$$\Psi(\lambda, U, V) = \lambda U + U Q_2(U,V),$$ where
$$\begin{array}{l} Q_2(U, V) = 
  U^2 +3 V^2 .\end{array} $$






Following Melrose \cite{M}we take advantage of the constraint $U \theta \leq 0$ to make an almost analytic extension of the $d \lambda$
integral to the upper (lower) half plane $\lambda + i \sigma, \sigma \geq 0$
accordingly as $\theta < 0\; (\theta  > 0)$. Applying the Gauss-Green theorem,\footnote{See Remark \ref{AAREM} for more details} 
$ e(t, q, \theta)$  \eqref{etqtheta}  is given by,
\begin{equation} \label{AA} |\theta|^{(n-1)} \int_{\C_+} \int_{\R} \int_{\R} e^{i \theta (U (\lambda +  i \sigma) + i/3U Q_2(U, V))}
\overline{\partial}_{\lambda + i \sigma} \tilde{A}((\lambda + i \sigma) |\theta|^{\frac{2}{3}}) d \lambda d \sigma dU dV, \end{equation} where 
$\tilde{As}$ is the analytic extension of $As$  \eqref{Asdef}.
A key point of the
almost analytic extension is that $\overline{\partial}_{\lambda + i \sigma} \tilde{A}((\lambda + i \sigma) |\theta|^{\frac{2}{3}}) $
vanishes to infinite order at $\sigma = 0$. This will enable us to integrate
by parts below in $dU$. 
 Moreover,  the Airy quotients have uniform
asymptotic expansions in the half plane when multiplied by functions vanishing to all orders at $\sigma = 0$.  
Thus the Airy symbol is classical in the half-plane. As in  (3.9) of \cite{M}, one concludes that 
\begin{equation}\label{UPSHOT}
e(t, q, \theta) = |\theta|^{(n-1)}  \int_{\sigma \theta \leq 0} \int_{\theta U \leq 0} e^{i |\theta| \left(U (\lambda + i \sigma) + \frac{1}{3}
U Q_2(U, V)\right)} a d \sigma d \lambda d U dV, \end{equation}
where (by Lemma \ref{ORDERLEM})  $a$ is a classical  symbol of order 1 in $\theta$ which is compactly supported near $U = V = \lambda + i \sigma  = 0$ and vanishes
to infinite order at $\sigma = 0$. 
 We now use \eqref{UPSHOT} as a new starting point to prove Lemma \ref{MLEM}.

\subsubsection{Proof of Lemma \ref{MLEM}}

\begin{proof}

We   integrate  \eqref{UPSHOT}  by parts in the $d U$ integral using
the fact that
$$\frac{\partial}{\partial U} \left( U (\lambda + i \sigma + Q_2(U, V)) \right) = \lambda + i \sigma 
+ d_U (U Q_2(U, V)) $$
has no zeros if $\sigma \not= 0.$
We form the integration by parts operator
$$L_U = \theta^{-1} \frac{1}{\lambda + i \sigma +  d_U U (Q_2(U, V))}  \frac{\partial}{\partial U}, $$
which we may use on $\{\sigma \not=0\}$. When $\sigma = 0$, and when the terms $\lambda + d_U [U Q_2(U, V)] = 0$ one picks up powers of $\sigma^{-1}$ on each partial integration but they are cancelled by the 
infinite order vanishing of the amplitude 
 at $\sigma = 0$.  Hence, we can integrate by parts any number  $M$ of times on the full domain $\sigma \geq 0$.   That is, we use that
 $|\frac{\sigma^M}{(i \sigma + \gamma)^M}| \leq 1$ where $\gamma
 = \lambda + d_U(U Q_2). $ Each partial
integration picks up a  boundary term at $U = 0$. At $U = 0$, the integral ceases to be oscillatory since the
whole phase has a factor of $U$ and vanishes when $U = 0$ and the 
boundary term equals \begin{equation} \label{BDYTERM}
e_0 (t, q, \theta) = |\theta|^{n-2}  \int_{\sigma \theta \leq 0} \int_{\R} \frac{a(\lambda + i \sigma, 0, V)}{\lambda + i \sigma +  (Q_2(0, V))} d \sigma d \lambda dV, 
  \end{equation} 
  Since $a$ has order 1, $e_0$ is a symbol of order $n-1$. Due to  the almost-analyticity of $a$, the amplitude vanishes to infinite order on
  $\sigma = 0$, so  the integral converges absolutely.
  
    Each additional partial integration introduces a factor of $\theta^{-1}$
  and gives a boundary term and an interior term. After $M$ repeated integrations
  by parts one ends up with a series of boundary terms of decreasing
  order in $\theta$ and an interior remainder term of order $|\theta|^{-M}$.  The leading boundary term \eqref{BDYTERM}  has a  classical expansion from that of the amplitude $a((\lambda + \sigma, 0, V) $,  and each lower
  order term has a similar expansion. The remainder after $M$ partial integrations is $|\theta|^{-M}$ times the symbol expansion of the
  corresponding amplitude, $(L^t)^M a((\lambda + \sigma, U, V) $.  
The order of $e$ is that of the leading boundary term $e_0$ and thus is $(n-1)$. Moreover, since the $dV$ integral
is over $[0, \epsilon]$ it follows that $e(t, \theta, q) \leq C \epsilon |\theta|^{n-1}.$
This completes the proof 
of  Lemma \ref{MLEM} for the $w_{1,G}$ terms.

\end{proof}

 \begin{rem} \label{AAREM}The formula  \eqref{AA}, or more explicitly,
 $$\begin{array}{l} 
  \int_{\C_+} \int_{\R} \int_{\R} e^{i \theta (U (\lambda +  i \sigma) + i/3U Q_2(U, V))}
\overline{\partial}_{\lambda + i \sigma} \tilde{A}((\lambda + i \sigma) |\theta|^{\frac{2}{3}}) d \lambda d \sigma dU dV\\ \\
=  \int_{\partial \C_+} \int_{\R} \int_{\R} e^{i \theta (U (\lambda +  i \sigma) + i/3U Q_2(U, V))} \tilde{A}((\lambda + i \sigma) |\theta|^{\frac{2}{3}}) d \lambda  dU dV, \end{array} $$
follows from the Gauss-Green formula (valid for any $C^1$ domain $D$
and $C^1$ function $f$),
$$\int_{\partial D} f(z) dz = - 2i \int \int_D \frac{\partial f}{\partial \bar{z}}(z) dL(z), $$
and from the fact that the phase $e^{i \theta (U (\lambda +  i \sigma) + i/3U Q_2(U, V))}$ is analytic in $\lambda$.

 The usual stationary phase method does not apply since the phase is
 degenerate and the $dU$ integral is over a half-line;
 the expansion comes from the boundary terms
 at $U = 0$. \end{rem}

 \subsection{\label{ALLwG} Completion of the proof of Lemma \ref{MLEMintro}}

In Lemma \ref{MLEM},  we proved that in the  Neumann case,  the singularity of $w_{1,G}(t)$ at $t =0$ is classical. 
To complete the proof of  Propositions \ref{MELROSE} in the Neumann case,  we must do the
same for the term  $w^{(2)}_{1,G}(t, q, q)$ of $w_{1,G}$
(see  \cite[(3.4)]{M}) as well as the terms $w_{2,G}^{\pm} $ (\cite[(2.26)]{M}  $(u_G)_c$ \cite[2.21]{M}. They are handled in a way very similar to $w_{1,G}(t)$ and we only sketch the changes in the proof.

In the  second term $w^{(2)}_{1,G}(t, q, q)$ of $w_{1,G}$  there is no $dr ds$ integral. Integrating out $dr$ and restricting to the boundary produces
\begin{equation} \begin{array}{lll} w^{(2)}_{1,G}(t, q, q)  & = & \int e^{i \alpha(0, t, q, \tau', \eta')
- i \alpha(0, 0, q, \tau', \eta') } 
As(\beta_0, \beta_0)  \;\;   d\tau' d \eta'. \end{array} \end{equation}This gives an integral of the form,
$$\int_{\R} As_{\pm} (\lambda |\theta|^{2/3}, \lambda |\theta|^{2/3}) a(\lambda) d \lambda$$
and we obtain a classical expansion using the classical expansion of $ As_{\pm} (\lambda |\theta|^{2/3}, \lambda |\theta|^{2/3}) $.

The terms $w_{2,G}^{\pm}$ have the same form as $w_{1,G}$ and the normality of the singularity at $t = 0$ is
proved in the same way. 

The term $u_c$ is obtained from $u_G$ \eqref{uG}  (\cite[(2.17)]{M}) by multiplying by $H(t)$ and cutting off in $x < 0$. 
 As explained
on page 270 of \cite{M}, the integral of $(u_G)_c$ over $z = z'$ is classical. We now verify that the boundary trace
$u_G^b$ on the diagonal of the boundary  is
classical. It is given by
\begin{equation} \label{uGb} u^b_G(t, q, q) = \int e^{i \Phi(0, t, q, s, \tau, \eta)  - i \Phi(0, 0, q, r, \tau, \eta)} c_G(0, t, q, s, r, \tau, \eta) d \tau d \eta
dr ds. \end{equation}
As above,
$\Phi(0, t, y, s, \tau, \eta) = \alpha(0, t, y,  \tau, \eta) + s |\eta|^{\frac{1}{3}} \beta (0,  y,  \tau, \eta)
+ \frac{1}{3} s^3 |\eta|$ and $c_G$ is a first order classical symbol supported in $|s|, |r| \leq \epsilon, |\tau^2 - |\eta|^2| \leq \epsilon \tau^2$ for some $\epsilon > 0$.  This form is similar to that of $w_{1,G}$ but does not have the Airy
amplitude. The proof of the normality of the singularity at  $t = 0$ of $w_{1,G}$ in Section \ref{GLSING} applies to it as well.

\subsection{Dirichlet boundary conditions}
From the point of view of glancing parametrix
constructions, the Dirichlet case is simpler than the Neumann case because
it is not necessary to invert the Neumann operator $\ncal$ in \cite[2.5]{M}. But $\ncal$ itself is
a Fourier-Airy integral operator of the same kind as in \cite[p. 263]{M}. See also \cite[Ch. X.5]{Tay2} and \cite{F} for
background on the Neumann operator.  The parametrices
have the same form and so we do not repeat the details of the parametrix
construction or of Propositions \ref{MLEMintro} and Proposition \ref{MLEM}. We only
discuss the change in Cauchy data due to the two normal derivatives.

  In the notation of \cite{M}, we
 now take $\partial_x \partial_{x''}$ and set $x = x'' = 0$. This changes the amplitude in several ways, all obvious
 from the facts that the phase and amplitude are symbols.
From the phase, differentiation brings down a factor of
$$(I)\;\;\;\;\partial_x \alpha(x, t, y, \tau', \eta')_{x = 0} , \;\; \partial_{x''} \alpha(x'', 0, y, \tau, \eta') |_{x'' = 0}. $$
Each derivative  raises the order by $1$ and the leading order term is given by two uses of the derivative
on the phase, raising the order by $2$.

If the derivative is placed instead on the amplitude, we first have the  term
$$(II)\;\;\; \partial_{x''}  \beta(x'', 0, y'', \tau,\eta') (- i r |\eta'|^{1/3}) |_{x''=0} $$
Finally we also have 
$$(III)\;\;\;\partial_{x} \partial_{x''} \frac{a  A_{\pm}(\beta) + b A'_{\pm}(\beta) }{A_{\pm}(\beta_0) }|_{x= x'' = 0}. $$

Thus, in the Dirichlet case, we  have
$$\begin{array}{l}  \;\partial_{x} \partial_{x''} w_{1, G}(t, q, q'')  \\ \\ = \int \int_0^{\infty}   e^{i \alpha(0, t, q, \tau', \eta') - i t' (\tau' - \tau) \mu 
+ i s \beta_0' |\eta'|^{\frac{1}{3}} + i \frac{s^3}{3} |\eta'|  }  e^{- i \alpha(0, 0, q'', \tau, \eta')  - ir \beta''
|\eta'|^{\frac{1}{3}} - i \frac{r^3}{3} |\eta'| }  dt'  d s dr d \tau d \tau'  d \eta'\\ \\ 
 \times \left[\partial_x \alpha(x, t, y, \tau', \eta')  + \partial_{x''} \alpha(x'', 0, y, \tau, \eta') + \partial_{x''}  \beta(x'', 0, y'', \tau,\eta') (- i r |\eta'|^{1/3})) 
\frac{a  A'_{\pm}(\beta_0) + b A_{\pm}(\beta_0) }{A_{\pm}(\beta_0) } \right] \\ \\+ \times \left[\partial_{x} \partial_{x''}  \frac{a  A'_{\pm}(\beta_0) + b A_{\pm}(\beta_0) }{A_{\pm}(\beta_0) } \right]
. \end{array}$$
The new amplitude has the same essential properties as that of the Neumann case since, as mentioned above,
the phases and amplitudes are symbols.  Thus,
as explained in \cite{M}, the argument of the Neumann case extends to the Dirichlet case with no essential
change.

\section{\label{PTWSECTb} Completion of the proof of Proposition \ref{MELROSE} }

Proposition \ref{MELROSE} asserts that the boundary trace $E_N^b(t, q, q)$  of  the Neumann wave kernel has a classical co-normal singularity
at $t = 0$. In the Dirichlet trace we take the normal derivative in each variable before restricting to the diagonal
of the boundary. 
In this section we complete the proof of Proposition \ref{MELROSE}. We briefly  discuss
the additional terms that were not handled in Section \ref{GLSING}. 
We also  discuss the modifications
when we apply a boundary pseudo-differential operator as in Proposition \ref{OpaPTWEYL}. 




\subsection{\label{EH} Elliptic and Hyperbolic terms}

In this section, we briefly discuss the elliptic and hyperbolic terms,
\begin{equation} \label{MICRODECOMPEH} \left\{ \begin{array}{l} w_{1,EH} = w^+_{1, H} + w^-_{1, H} + w_{1, E} \\ \\
w^{\pm}_{2, EH} = w^{\pm, +}_{2, H} + w^{\pm, -}_{2, H} + w^{\pm}_{2, E} \end{array} \right.\end{equation}
of \eqref{MICRODECOMP}. For simplicity of notation we denote any of the terms by $e_{EH}(t)$.
Parallel to Lemma \ref{MLEMintro} and Lemma \ref{MLEM} we assert the following:

\begin{lemma} \label{EHLEM} On the diagonal of $\partial M \times \partial M$, the non-glancing part $w_{EH}$ of the parametrix has a normal singularity at $t = 0$, i.e.
\begin{equation} \label{ehdef} w_{EH}(t, q, q) = \int e^{i t \theta} e_{EH}(t, \theta, q) d \theta, \end{equation}
where  $e_{EH}(t, \theta, q)$  is a classical symbol in $\theta$ of order $n-1 $ when $ q \in \partial M$.

 \end{lemma}
 
 We do not give a proof of this statement because it is already known. 
Parametrices for the wave group on a manifold with boundary cut off from the glancing directions are given in \cite{Ch2}.   Away from the glancing (tangential) directions,
$E^b_B(t, q, q)$ is a Fourier integral kernel whose symbol is computed in \cite{HeZ}. It is  further
discussed in \cite{HeZ,HHHZ} (see also \cite{SmS95}).

\subsubsection{Pseudo-differential cutoffs to the non-glancing region} 
In the first part of  Proposition \ref{PTWEYL},  we apply a boundary (semi-classical) pseudo-differential operator $Op_h(a)$ under
the integral sign in both variables for  the parametrices for the wave kernel and their boundary traces.  When the symbol vanishes near the glancing direction, there is a
standard Fourier integral parametrix and one may apply 
the  pseudo-differential operator $Op_h(a)$  to the oscillatory integral.
 By the  ``fundamental asymptotic
expansion Lemma"  of as  \cite{Tay2} (Section VIII \S 7), application of $Op_h(a)$   changes the 
amplitude to another amplitude with the same symbolic properties. Thus, in  the elliptic or hyperbolic regions, where the oscillatory integral satisfies the assumptions (2.3)-(2.4) of Taylor (loc.cit.), 
the oscillatory integrals are standard ones. Indeed, such cutoffs essentially lead back to
the statements of Section \ref{EH}.

\subsection{\label{PSIDOGL} Pseudo-differential cutoffs to the non-glancing region} 

For the glancing term of Proposition \ref{OpaPTWEYL} we need to apply $Op_h(a)$ on the left side and right side to $w_{G}(t, q, q')$ in
\eqref{w1} for  $q= (0, y), q'' = (0, y'') \in \partial M$, and then set $q = q'$. For future reference, we state
the generalization of Lemma \ref{MLEMintro} for this modification as follows:

\begin{lemma} \label{PSIDOMLEM} On the diagonal of $\partial M \times \partial M$, 
\begin{equation} \label{COedef} Op_h(a) w_{G}(t, q, q') |_{q' = q}= \int e^{i t \theta} e_{a}(t, \theta, q) d \theta, \end{equation}
where  $e_{a}(t, \theta, q)$  is a classical symbol in $\theta$ of order $n-1 $ when $ q \in \partial M$.

 \end{lemma}

\begin{proof}

We are applying
semi-classical pseudo-differential operators on a manifold without boundary
to an oscillatory integral

\begin{equation} \label{w1G21b} \begin{array}{lll} w_{1, G}(t, (0,y), (0,y''))  &=& \int \int_0^{\infty}   e^{i \left(\alpha(0, t, y, \tau', \eta')  -  \alpha(0, 0, y'', \tau, \eta') \right) - i t' (\tau' - \tau) \mu }\\&&\\&&e^{
 i s \beta_0' |\eta'|^{\frac{1}{3}} + i \frac{s^3}{3} |\eta'|  }  e^{  - ir \beta''
|\eta'|^{\frac{1}{3}} - i \frac{r^3}{3} |\eta'| }
\frac{a  A'_{\pm}(\beta_0) + b A_{\pm}(\beta_0) }{A_{\pm}(\beta_0) }\\ &&\\&& 
 dt'  d s dr d \tau d \tau'  d \eta'. \end{array} \end{equation}
where $\beta_0' =\beta(0, q, \tau, \eta')$ and $\beta'' = \beta(0, q'', \tau, \eta')$ are defined in \eqref{beta}The terms of the phase
$\alpha(0, t, y, \tau', \eta')  -  \alpha(0, 0, y'', \tau, \eta')$ and $t'(\tau- \tau) \mu$ are classical. 
Taking into account Property  (2) of the list in Section \ref{GTSECT}, 
$\beta_0$ is independent of $q \in \partial M$, and has the form $\beta(0, q, \tau, \eta) = (\tau^2 - |\eta|^2) |\eta|^{-4/3}$,

Hence
$Op_h(a)$ in either $q$ or $q'$ is the  application of a pseudodifferential
operator with a Fourier-Airy integral operator with a classical phase but
with an  Airy amplitude, which is a non-classical symbol.
The fundamental asymptotic expansion lemma is essentially the stationary
phase method, and it applies to this composition in the $(q, q')$ variables. 
Thus, after application of $Op_h(a)$ on either side we obtain a new Fourier
Airy integral operator on $\partial M$ with the same phase but a new amplitude,
with the same properties as those of $w_{1,G}$ and of the same order.
The composition of Airy operators and the symbol expansion is discussed in detail in \cite{M78}. The rest of the argument proceeds as in the proof
of Proposition \ref{PTWEYL}.



\end{proof}

\section{\label{PTWSECT} Proof of the pointwise Cauchy data Weyl laws of Propositions \ref{PTWEYL} and \ref{OpaPTWEYL}}

In this section, we complete the proofs of Proposition \ref{MLEMintro} and Proposition \ref{MLEM}. They
follow from a standard cosine Tauberian argument and Proposition \ref{MELROSE} once it is converted into
 dual statements about convolutions of spectral measures with special test functions.
In addition, we  need to justify the statement that   
    the contribution of the 
 glancing terms
to the first two terms of the expansions of Proposition  \ref{OpaPTWEYL}   are of order $O(\epsilon)$ as $\epsilon \to 0$

\subsection{Proof of Proposition \ref{PTWEYL}}

Following the standard route of Fourier Tauberian theorems, we first observe that 
\begin{equation} \label{Sqt}  S_q(t) = \fcal_{\lambda \to t}\; d_{\lambda} \Pi^b_{[0, \lambda]}(q,q), \end{equation}
where $S_q$ is defined in \eqref{Seq} and $\Pi^b_{[0, \lambda]}(q,q)$ is defined in \eqref{BTSPEC}.
To determine the co-normal expansion of the singularity in Proposition \ref{MELROSE} at $t = 0$, we study the dual problem
\begin{equation} \label{Seq} S_q(\lambda, \rho) = \rho * d_{\lambda} \Pi^b_{[0, \lambda]}(q,q) =  \int_{\R} \hat{\rho} (t) \;S_q(t) e^{i t \lambda} dt \end{equation}
where $\rho  \in \scal(\R)$ (Schwartz space) with $\hat{\rho} \geq 0$,  $\hat{\rho} \in C_c^{\infty}(\R)$
even, satisfying $\int_{\R} \rho dx = 1$ and with   supp  $\hat{\rho}$ contained
in a  sufficiently small neighborhood $[-\epsilon, \epsilon]$ of $t = 0$ so that $t = 0$ is the only singularity 
of $E_B^b(t, q, q)$ in supp $\hat{\rho}.$  Thus,
\begin{equation}\label{eq:BFSSb1} 
S_q(\lambda, \rho) =\frac{\pi}{2}\sum_j (\rho(\lambda-\lambda_j) + \rho(\lambda+\lambda_j)) |\phi_j^b(q)|^2.
\end{equation}

To prove Proposition \ref{MELROSE} and Proposition \ref{PTWEYL} it suffices to prove 

\begin{lemma} \label{SF2} Let $(M, g)$ be a compact Riemannian manifold with concave boundary.  If supp  $\hat{\rho}$ is sufficiently close to $t = 0$, then $S_q(\lambda, \rho)$ is a semi-classical Lagrangian distribution whose asymptotic expansion in
in the Neumann, resp.  the Dirichlet, case 
has the form,

\begin{equation}\label{eq:BFSSb} 
S_q(\lambda, \rho) = \left\{ \begin{array}{l}
 C_n \lambda^{n-1} + Q_N(q) \lambda^{n -2 } + O(\lambda^{n-3}), \; \rm{Neumann}, \\ \\
 C_n \lambda^{n+1} + Q_D(q) \lambda^{n } + O(\lambda^{n-1}), \; \rm{Dirichlet} \end{array} \right.
\end{equation}
where $C_n = \frac{\omega_n}{(2 \pi)^n}$, $C_n'$  is a constant depending only on the dimension and $Q_{D,N} (q)$ is a local geometric invariant of $\partial M$, equal to a dimensional constant times the mean curvature   in the case of Dirichlet boundary conditions. \footnote{It is almost certainly given by a dimensional constant times the mean curvature in the case of Neumann boundary conditions by the same invariance theory argument as in \cite{O79}. }\end{lemma}

Above,  $\omega_n$ is the volume of the unit ball in $\R^n$ and $Q_D, Q_N$ are defined in \eqref{QDN}. The extra
power in the Dirichlet case is due to the normal derivatives, since we chose
not to use semi-classical normal derivatives in   \eqref{SCTr}. 

\begin{proof} (Sketch)

There are
two independent aspects of the Proposition. The first is to prove that $S_q(\lambda, \rho)$ has a complete asymptotic
expansion when $\hat{\rho}$ has support sufficiently close to $t = 0$. By Fourier transform methods, this
is equivalent to normality of the singularity  of $E_B^b(t, q, q)$ at $t = 0$
of Proposition \ref{MELROSE}.  If we use the microlocal decompositions \eqref{MICRODECOMP}, then $S_q(t)$
and 
$S_q(\lambda, \rho)$ break up into corresponding terms $S_{q, G}(\lambda, \rho), S_{q,EH}(\lambda, \rho)$ corresponding to $e(t,\theta, q)$
\eqref{edef} and $e_{EH}(t, \theta, q)$ of Lemma \ref{EHLEM}. Comparing \eqref{edef} and \ref{Seq}, we see that
\begin{equation} \label{SqG} S_{q, G}(\lambda, \rho) = \int_{\R} \int_{\R} \hat{\rho}(t) e^{i t \lambda} e^{i t \theta} e(t, \theta, q) d \theta dt, \end{equation}
and similarly for $S_{q, EH}$ with $e_{EH}(t, \theta,q)$ replacing $e(t, \theta,q). $ Changing variables $\theta \to \lambda \theta$ produces a semi-classical oscillatory integral with phase $t(1 + \theta).$ The phase is non-degenerate
with only one critical point at $\theta = -1, t = 0$, and by stationary phase, one has complete asymptotic expansions
\begin{equation}\label{Se} \left \{ \begin{array}{l} S_{q, G}(\lambda, \rho) \simeq  \hat{\rho}(0) \; e(0, -  \lambda, q) + \cdots \\ \\
S_{q, EH}(\lambda, \rho) \simeq \hat{\rho}(0)\; e_{EH}(0, -\lambda, q) + \cdots,
\end{array} \right. \end{equation}
where we omit the lower order terms arising from the Hessian operator expansion of stationary phase for brevity.
In particular, $S_{q, G}, S_{q, EH}$ are symbols of order $n -1$ in $\lambda$. In the case of $S_{q,G}$, all
terms are of order $\epsilon$ as in the proof of Lemma \ref{MLEM}, and so is the stationary phase remainder,
since they are all given as integrals in $dV$ over $[0, \epsilon]$.

  The second aspect of Lemma \ref{eq:BFSSb}  is the calculation of the coefficients, which is not actually needed for the proof of Theorem \ref{SoZthm}. Hence, we only provide some references.
The calculation of $Q_D$ is  essentially due to S. Ozawa \cite[Theorem 1, Proposition 3]{O79}, who used the Hadamard variational method to 
calculate the  coefficients in the small $t$ expansion of the  Cauchy data of the heat kernel.  See also \cite{Mi} for special cases and corrections. The coefficients above can be calculated by subordination
of the wave kernel to the heat kernel, once it is proved that  $S_q(\lambda, \rho) $ admits an expansion.
  A formal calculation of the first two coefficients is also
given in (4.3) of  \cite{BSS} in dimension two.

\end{proof}

To complete the proof of Proposition \ref{PTWEYL}, we apply a  standard cosine Tauberian theorem  \ref{ITT} of \cite{I80} (see
Section \ref{TAUBERAPP}, 
 or \cite[Lemma 17.5.6]{HoI-IV}), where we let  and $T =1$. \footnote{$\rho$ is denoted $\hat{\beta}$ in Theorem \ref{ITT}.}
\begin{equation} \label{TAUBER1} \Pi_{[0, \lambda]}^b(q,q) =  \rho *  \Pi_{[0, \lambda]}^b(q,q) + O(\lambda^{n-1}). \end{equation}
There is one integration in $\lambda$ by comparison with \eqref{Seq}, raising the orders by 1.

\subsection{\label{PTGENSECT} Pointwise Weyl laws for cutoff Cauchy data}

 To prove Proposition \ref{EPLEM} we  need the following
 generalized pointwise  Weyl law  for Cauchy data of eigenfunctions in Section \ref{SUPSECT}. We denote by $Op_h(a)$ a semi-classical
pseudo-differential operator on $\partial M$. We refer to \cite{Zw,HZ,tz1,HHHZ} for background.

\begin{prop} \label{OpaPTWEYL} Let $(M, g)$ be a compact Riemannian manifold with concave boundary. There is a constant $C$ depending only on $(M,g)$ so that if $Op_h(a)$ is a
semi-classical zero order pseudodifferential operator  on $\partial M$ with principal symbol $a_0(q, \eta)$ vanishing in an $\epsilon$-neighborhood of the glancing set, then for $q \in \partial M$,  in the Neumann case, 
\begin{equation}\label{z.8}
\sum_{\la_j\in [0,\la]}| Op_h(a)  \phi^b_j(q)|^2
= \la^{n}\int_{B_q^*\partial M}|a_0(q,\xi)|^2 (1 - |\xi|^2)^{- \half} \, d\xi +O_{a, \epsilon}(\la^{n-1}).
\end{equation}
In the Dirichlet case,
\begin{equation}\label{z.8D}
\sum_{\la_j\in [0, \la]}| Op_h(a)  \phi^b_j(q)|^2
= C\la^{n+ 2}\int_{B_q^*\partial M}|a_0(q,\xi)|^2 (1 - |\xi|^2)^{ \half} \, d\xi +O_{a, \epsilon}(\la^{n +1}).
\end{equation}

On the other hand, if the principal symbol is supported in an $\epsilon$-neighborhood of the
glancing set, then
in the Neumann case, 
\begin{equation}\label{z.8n}
\sum_{\la_j\in [0,\la]}| Op_h(a)  \phi^b_j(q)|^2
= O( \epsilon \; \la^{n}).
\end{equation}
In the Dirichlet case,
\begin{equation}\label{z.8N}
\sum_{\la_j\in [0, \la]}| Op_h(a)  \phi^b_j(q)|^2
= O (\epsilon \; \la^{n +2}).
\end{equation}

\end{prop}


\begin{proof}

The proof is very similar to that of Proposition \ref{PTWEYL}.  It only
differs in that we apply a pseudo-differential operator on the boundary first. 
Proposition
\ref{OpaPTWEYL}  is again proved using a cosine Tauberian theorem applied to the Weyl sums,
\begin{equation} \label{OpSp} N_a(\lambda, q): =
\sum_{\la_j\in [0,\la]}| Op_h(a)  \phi^b_j(q)|^2.
\end{equation} We then study the convolution,

\begin{equation} \label{Seqa} S_a(\lambda, \rho) = \rho * d_{\lambda} N_a(\lambda, q)  = \int_{\R} \hat{\rho} (t) \;S_a(t,q,q) e^{i t \lambda} dt \end{equation}
where 
\begin{equation} \label{Sata}  S_a(t, q, q):= \sum_j | Op_h(a)  \phi^b_j(q)|^2 \cos t \lambda_j = \fcal_{\lambda \to t} \;
d_{\lambda} N_a(\lambda, q),  \end{equation}
and where
$\rho  \in \scal(\R)$ is as before. Thus,

\begin{equation}\label{eq:BFSSb12} 
S_a(\lambda, \rho) =   \frac{\pi}{2}\sum_j (\rho(\lambda-\lambda_j) + \rho(\lambda+\lambda_j)) |Op_h(a) \phi_j^b(q)|^2.
\end{equation}

\begin{lemma} \label{rhoSaLEM} 

\noindent{(i)} If $Op_h(a)$ is microsupported in the complement of an $\epsilon$-neighborhood
of the glancing set, then

\begin{equation}\label{eq:BFSSb2} 
S_a(\lambda, \rho) = \left\{ \begin{array}{l}  

\la^{n-1}\int_{B_q^*\partial M}|a_0(q,\xi)|^2 (1 - |\xi|^2)^{- \half} \, d\xi +O_{a, \epsilon}(\la^{n-2}). \; \rm{Neumann}, \\  \\
 C\la^{n+ 1}\int_{B_q^*\partial M}|a_0(q,\xi)|^2 (1 - |\xi|^2)^{ \half} \, d\xi +O_{a, \epsilon}(\la^{n}). \; \rm{Dirichlet} \end{array} \right.
\end{equation}\bigskip

\noindent{(ii)} If $Op_h(a)$ is microsupported in an $\epsilon$-neighborhood
of the glancing set, then

\begin{equation}\label{eq:BFSSbg} 
S_a(\lambda, \rho) = \left\{ \begin{array}{l}  O(\epsilon
\la^{n-1})\ \; \rm{Neumann}, \\  \\
 O(\epsilon \; \la^{n+ 1}). \; \rm{Dirichlet} \end{array} \right.
\end{equation}

\end{lemma}

\begin{proof} (Sketch) 
The existence of the expansions follows from the normality of the singularity of $S_a(t)$, which is
proved in Lemma \ref{PSIDOMLEM} in Section  \ref{PSIDOGL}. As in Lemma \ref{SF2}, the expansions are easily
derived from those of  $e_a$ in Lemma \ref{COedef}  and the corresponding symbol for the elliptic or hyperbolic
regions. The calculations of the principal coefficients in the non-glancing region
are from \cite{HZ,HeZ,HHHZ}.

In the glancing region, $e_a$ is essentially the same as $e(t, \theta, q)$ \eqref{edef}, as discussed in Section
\ref{PSIDOGL}, and \eqref{Se} remains valid.   As in that setting, we use the microlocal decompositions \eqref{MICRODECOMP} to express  $S_a(t)$
and 
$S_a(\lambda, \rho)$ as a sum of terms $S_{a, G}(\lambda \rho), S_{a,EH}(\lambda, \rho)$  as in  Lemma \ref{EHLEM}, and find that
$$S_{a, G}(\lambda, \rho) = \int_{\R} \int_{\R} \hat{\rho}(t) e^{i t \lambda} e^{i t \theta} e(t, \theta, q) d \theta dt, $$
and similarly for $S_{q, EH}$ with $e_{EH}(t, \theta,q)$ replacing $e(t, \theta,q). $ Changing variables $\theta \to \lambda \theta$ produces a semi-classical oscillatory integral with phase $t(1 + \theta).$ The phase is non-degenerate
with only one critical point at $\theta = -1, t = 0$, and by stationary phase, one has complete asymptotic expansions
\begin{equation}\label{SaG} \left \{ \begin{array}{l} S_{a, G}(\lambda, \rho) \simeq  \hat{\rho}(0) \; e(0, -  \lambda, q) + \cdots \\ \\
S_{a, EH}(\lambda, \rho) \simeq \hat{\rho}(0)\; e_{EH}(0, -\lambda, q) + \cdots,
\end{array} \right. \end{equation}
where we omit the lower order terms arising from the Hessian operator expansion of stationary phase for brevity.
In particular, $S_{q, G}, S_{q, EH}$ are symbols of order $n -1$ in $\lambda$. In the case of $S_{q,G}$, all
terms are of order $\epsilon$ as in the proof of Lemma \ref{MLEM}, and so is the stationary phase remainder,
since they are all given as integrals in $dV$ over $[0, \epsilon]$.

\end{proof}

We again  apply the  cosine Tauberian theorem \eqref{ITT}  with $\rho = \hat{\beta}$ and $T = 1$ to complete the proof.

 Another way to check the $O(\epsilon)$ is as follows: In  the notation of \eqref{z.8}, where $\eta \in B_q^* \partial M$ is the projection of a unit covector
 $\xi \in T_q^*M$ to $T_q \partial M$ making an angle of $\leq \epsilon$.  The sector corresponds to $\sqrt{1 - |\eta|^2} \leq \epsilon $ in $B^*_q M$ or to $ \sqrt{\tau^2 - |\eta|^2} \leq \epsilon \tau^2 $ in the homogeneous model.  Hence, using polar coordinates $|\xi|  =r$ in $B^*_q M$, the integral  is bounded by
$\int_{\sqrt{1 - \epsilon^2}}^1 (1- r)^{-\half} dr = O(\epsilon) $ in the Neumann case and
by $\int_{\sqrt{1 - \epsilon^2}}^1 (1- r)^{\half} dr $ in the Dirichlet case. 
It follows that the terms obtained as boundary values in the integration by parts are also of order $\epsilon$.

 \end{proof}

\section{\label{SUPSECT} Sup norm bounds for Cauchy data: Proof of Theorem \ref{SoZthm}}


We now complete the proof of Theorem \ref{SoZthm}, following
the outline of  that in \cite{SZ1}. The theorem follows from Proposition \ref{EPLEM}, but
in fact it suffices to prove a somewhat weaker statement, Lemma \ref{1.11} below. After
proving this Lemma and Theorem \ref{SoZthm} we prove the stronger asymptotic result
of Proposition \ref{EPLEM}.

We  define the boundary  billiard  loop-length function  on the unit co-ball bundle $B^* \partial M$  by
\begin{equation}\label{M1}
L^*(q, \eta)= \left\{ \begin{array}{l} \inf \{t>0: \,\Phi^t(q, \xi) =q, \; \rm{if \; q\; is \; a \; loop\; point}  \},\\ \\
\infty,  \;\rm{if\; no\; such \; t \; exists}. \end{array} \right. 
\end{equation}
where $\xi \in S^*_q M$ is the unit co-vector projecting to $\eta \in B^*_q \partial M$ and where $\Phi^t$ is the billiard flow on $T^* M$.  $L^*$ is  homogeneous
of degree zero,  so it is natural to consider the restriction
of $L^*$ to unit covectors to $M$ along $\partial M$. 
Note that  $L^*$  is a lower semicontinuous function, or equivalently that
the function $1/L^*(q, \eta)$, which is defined to be zero when $L^*(x,\xi)=+\infty$,
is an upper semicontinuous function.

We  introduce a cutoff $\hat{\rho} \in C_0^{\infty} (\R)$, which as above is  a positive even  function   such that $\hat\rho$ is identically $1$ near $0$, has support in $[-1,1]$ and is decreasing on $\R_+$.  We also define $\rho_T$ by
$\hat{\rho}_T(t) = \hat{\rho}(\frac{t}{T})$, so that   $\supp \hat{\rho}_T  \subset  (-T,T)$. 
To prove Theorem \ref{SoZthm}   it suffices to prove
\begin{lemma} \label{1.11} If   $|{\mathcal L}_q|=0$, then for all $\epsilon > 0$  there exists a neighborhood $\ncal(q, \e) \subset \partial M$ 
and a time $T_0(\epsilon)$ so that
for $T \geq T_0(\epsilon)$,
\begin{equation}
\sum_{j=0}^\infty
\rho\bigl(T(\la-\la_j)\bigr)|\phi^b_j(q')|^2 \le \e \la^{n-1}, \quad
\text{if } \,
 q'\in {\mathcal N}(q,\e), \, \,
\text{and } \, \, \la\ge \Lambda.
\end{equation}
\end{lemma}

Before proving  Lemma \ref{1.11}, we show that it implies  Theorem \ref{SoZthm}. 

Indeed,  we observe that for fixed $q \in
\partial M$ and any $\varepsilon>0$,  one
can find a neighborhood $\mathcal{N}_\varepsilon(q)$ of $q$ and an
$\Lambda_\varepsilon(q)$ so that when $\lambda\ge
\Lambda_\varepsilon(q)$ and $y\in \mathcal{N}_\varepsilon(q)$ we
have $|R(\lambda,y)|\le \varepsilon \lambda^{n-1}$.  This implies
that $|\phi_j^b(y)|\le \varepsilon \lambda_j^{(n-1)/2}$ if $y\in
\mathcal{N}_\varepsilon(q)$ and $\lambda\ge
\Lambda_\varepsilon(q)$. Since $M$ is compact and since the open sets
$\{\mathcal{N}_{\varepsilon}(q)\}$  form open cover of $M$, we may
choose a finite subcover and extract the largest
$\Lambda_{\varepsilon}(q)$. For this $\Lambda_\varepsilon$, Lemma \ref{1.11} gives
\begin{equation} \label{phiep} |\phi_j^b(q)|\le \varepsilon\lambda_j^{(n-1)/2},
\quad \lambda_j\ge \Lambda_{\varepsilon}.  \end{equation} Since  $\Lambda_{\varepsilon}$ depends only on $\varepsilon$,
it follows that $\sup_{q \in \partial M} |\phi_j^b(q)| = o(\lambda_j^{(n-1)/2})$, as stated in Theorem \ref{SoZthm}.


\subsection{Proof of Lemma \ref{1.11}}

\begin{proof}

We  consider the smoothed restricted Weyl sum \eqref{eq:BFSSb}
\begin{equation} \label{rho} \begin{array}{lll} S_{q}(\lambda, \rho_T) := \rho_T * d_{\lambda} \Pi_{[0, \lambda]} (q,q) & = & \frac{1}{T} \int_{\R}   E_B^b(t, q, q)
\hat{\rho}(\frac{t}{T}) e^{ - i t \lambda} dt \\ &&\\
& = & \sum_j (\rho(T(\lambda-\lambda_j)) + \rho(T(\lambda + \lambda_j))  |\phi_j^b(q)|^2.
\end{array} \end{equation} As in \cite{SZ1}, the $\rho(T(\lambda + \lambda_j))$ term contributes $\ocal(\lambda^{-M})$
for all $M > 0$ and therefore may be neglected. To prove Lemma \ref{1.11} it suffices to show that for any $T$,
\begin{equation}\label{z.4}
|S_{q'}(\lambda, \rho_T)|\le \e\la^{n-1}+O_T(\la^{n-2}),  \quad q' \in \ncal(q, \e).
\end{equation}

If $|{\mathcal L}_q|=0$, then
$${\mathcal L}^T_q=\bigl\{\xi\in S^*_{q, in} M: \, \,
\Phi^t(q,\xi)\in S^*_qM \, \, \text{for some } \,
t\in [-T,T]\backslash \{0\}\bigr\}$$
is closed and of measure zero. For a given $T > 0$, we  can therefore  construct a  pseudodifferential cutoff
\[
 \chi_T(q, D): L^2(\partial M) \to L^2(\partial M) \] with the
property that $\chi_T$ is microsupported in the set where $L^*(q, \eta)
>> T$ and $1 - \chi_T$ has small support.  Thus, given  $\e_0>0$, we can find a
$\chi_T \in C^\infty(B^* \partial M)$  so that
\begin{equation}\label{3.1}
0\le \chi_T \le 1, \quad \int_{B_q^* \partial M}  (1 - \chi_T)\, d\sigma <\e_0, \quad
{\mathcal L}_x^T \cap \supp \chi_T = \emptyset.
\end{equation}  In fact, we may  construct $\chi_T(q, \eta) \in C_0^{\infty}(B^* \partial M)$ so that
\begin{equation} \label{SMALL} \int_{B_q^* \partial M} (1 - \chi_T)(q, \eta)d\sigma(\eta)\le 1/T^2, \end{equation} and 
$$|L^*(q, \eta)|\geq T, \quad \text{on} \, \, \text{supp}\; \chi_T.
$$ 
 We also denote by $\chi_T(q, D) = Op_{\hbar}(\chi_T)$ its quantization as a semi-classical
pseudo-differential operator on $\partial M$ with the sequence of Planck constants  $\hbar_j = \lambda_j^{-1}$. That is, $\chi_T(q, D)$ is a semi-classical pseudo-differential
operator with symbol $\chi_T(q, \eta)$ (see e.g.  \cite{Zw} for background on semi-classical pseudo-differential operators, and \cite{ctz,tz1, TZ2, HHHZ}  for background in the context of this article.)

We then use a semi-classical microlocal cutoff $\chi_h^{\epsilon} (x, D)$ on $\partial M$ to a conic $\epsilon$ neighborhood of the
glancing set to decompose   $$\Pi_{[0,\lambda]}^b(q,q)=  \Pi_{[0,\lambda]}^{\leq \epsilon, b} +  \Pi_{[0,\lambda]}^{\geq \epsilon, b}$$ into  a glamcing term $\Pi_{0, \lambda]}^{\leq \epsilon, b} = \chi_h^{\epsilon}(x, D) \Pi_{[0, \lambda]}^b(q,q') |_{q = q'}$ microlocalized to an angle
$\leq \epsilon$ around the glancing set and  the complementary  hyperbolic term  $ \Pi_{[0,\lambda]}^{\geq \epsilon, b}$
using $(I - \chi_h^{\epsilon}(x, D))$.
We further decompose $\Pi_{[0,\lambda]}^{\geq \epsilon, b}$ by using a microlocal cutoff to a neighborhood of the union of 
 loops
of length $\leq T$ cut out; (iii) the complementary hyperbolic term mirolocalized to loops of length $\leq T$:

\begin{equation} \label{MPU} \begin{array}{lll}  \Pi_{[0,\lambda]}^{\geq \epsilon, b} (q,q)  & =  &
[(\chi_T(q, D) + (I - \chi_T(q, D))  \circ \Pi_{[0,\lambda]}^{\geq \epsilon, b}  \circ\\&&\\&& \circ (\chi_T(q, D)  + (I - \chi_T(q, D))](q,q). \end{array} \end{equation}

There are two types (I, II) of terms among the four in \eqref{MPU}. The first type I  (of which there are three terms) has at least one factor of $\chi_T$ (on either
side of $\Pi_{[0,\lambda]}^b$). The second type (of which there is just one term)  has the form
\begin{equation} \label{TWO}  (I - \chi_T(q, D)))  \circ \Pi_{[0,\lambda]}^b  \circ  (I - \chi_T(q, D)))](q,q). \end{equation}
The first type of term can be dealt with entirely by wave front set considerations. The second
is more complicated but can be dealt with by Lemma \ref{MLEM}.

We break up the analysis into one for short times and one for the rest.  We  fix an even function
$$\beta\in C^\infty_0(\R) \;\; \mbox{which equals one on}\; [-2\delta,2\delta]. $$
 The analysis for small times $[-\delta, \delta]$ involves the short
time parametrix  and pointwise Weyl laws discussed in \S \ref{PTWSECT} and \S \ref{PTGENSECT}  while for times $|t| \geq \delta$ it is not necessary to construct a parametrix.

\subsection{Glancing terms}

\begin{lemma}\label{GLANCINGLOOP}

Let $\chi$ be the pseudo-differential operator  cutting off to $\epsilon$ neighborhood of the  glancing set. Then  
$$\sum \rho(\lambda-P) |\chi e_j(q)|^2 \le C\epsilon \lambda^{n-1}.$$
\end{lemma}

\begin{proof} The statement follows from  Lemmas \ref{SF2} or \ref{rhoSaLEM}.
The left hand side equals
$$\iint \Hat \rho(t) e^{-it\theta} e(t,q,\theta) e^{it\lambda} dt d\theta,$$
which is \eqref{SqG} (up to a sign), and by Lemma \ref{MLEM} and \eqref{BDYTERM},   $$e(0,   \lambda, q) = 
 O(\epsilon \lambda^{n-1}) + O_\epsilon(\lambda^{n-2}).$$

\end{proof}

\subsection{Terms of both type I and II for $|t| \leq \delta$}

Here we do not gain from using the cutoff $\chi_T$ since $\delta$ is less than the minimal loop length and
the only singularity is at $t = 0$. 
We must use the analysis of the singularity at $t = 0$ in \S \ref{PTWSECT} and \S \ref{PTGENSECT}.

\begin{slem} \label{slembeta}  For
$T\ge 1$,
$$\left|\frac1{2\pi T}\int \beta(t)\rho(t/T)
E_B^b (t, q,q) \, e^{-it\la}\, dt\right|
\le	 CT^{-1}\la^{n-1}. $$

\end{slem}

\begin{proof}
 For
$T\ge 1$,
$$\left|\frac1{2\pi T}\int \beta(t)\rho(t/T) \, e^{it\tau}\, dt\right|
\le C_NT^{-1}(1+|\tau|)^{-N}, \quad N=1,2,3,\dots.$$
Hence,
\begin{multline*}\left|\frac1{2\pi T}\int \beta(t)\rho(t/T)
E_B^b (t, q,q) \, e^{-it\la}\, dt\right|
\\
\le CT^{-1}\sum_{j=0}^\infty (1+|\la-\la_j|)^{-n-1}(\phi^b_j(q))^2 
= O(T^{-1}\lambda^{n-1}).
\end{multline*}

Here, we used that
$$\begin{array}{lll} \sum_{j=0}^\infty (1+|\la-\la_j|)^{-n-1}(\phi^b_j(q))^2 &  = & \int (1 + (\lambda - \mu)^{-n-1} d_{\mu}
\Pi^b_{[0,\mu]}(q,q)  \\&& \\
 &  = & \int (1 + (\lambda - \mu)^{-n-1} d_{\mu} (\mu^n + R(\mu, q) )
  \\&& \\
 &  = & n \int (1 + (\lambda - \mu)^{-n-1} \mu^{n-1} +   \\&& \\
 &  + &(n-1) \int (1 + (\lambda - \mu)^{-n-2}  | R(\mu, q)| d\mu\\&&\\
& \leq & 2 n \int (1 + (\lambda - \mu)^{-n-2}  \mu^{n-1} d\mu = O(\lambda^{n-1}).
\end{array}$$

In the  last line we used  the remainder estimate in the 
 pointwise  Weyl law in   Proposition \ref{PTWEYL}. \end{proof}

\subsection{Terms of type I for $|t| \geq \delta$}

In this section we prove,

\begin{slem} \label{TYPEI}
\begin{equation}\label{3.3}
\frac1{2\pi T}
\int \bigl(1-\beta(t)\bigr) \rho(t/T) \, \bigl( \chi_T(q, D) \circ E_B^b\bigr)(t, q,q) \,
e^{-it\la} \, dt =O_{T}(1).
\end{equation}
\end{slem}

\begin{proof} The estimate follows immediately from the fact that


\begin{equation} \label{SMOOTH} E_B^b(t)\chi_T(q, D)^*(q,q),\;\;  \chi_T(q, D) \circ  E_B^b(t,q,q) \in C^{\infty}(0,
T). \end{equation}
To see this, let $d_{\partial M}$ denote the Riemannian distance along $\partial M$,
and let $T^{(j)}(q)$ be the length of the interior billiard trajectory corresponding corresponding to the billiard orbit starting at $q$ and ending at $\beta^j(q)$.
By assumption,
$$\min_{\eta\in \supp \chi}d_{\partial M} (q,\beta^k(q, \eta)) )>0, \; \rm{if}\;\; \sum_{j = 0}^k T^{(j)}(q, \eta)) < T $$
and so there must be a neighborhood ${\mathcal N}$ of $q$ in $\partial M$  so that if $q' \in {\mathcal N}$ then
$$\{\beta^j(q', \eta) ,  0 < j \leq k \; \mbox{s.th.}\;  \sum_{j = 0}^k T^{(j)}(q', \eta)) < T \}\notin {\mathcal N}, \quad \text{if } \, \,
\eta\in \supp \chi.$$
Therefore by \eqref{WF1}-\eqref{WF2},
\begin{equation}\label{3.2}
\bigl(\chi_T(q, D) \circ E_B^b\bigr)(t, q,q)\in C^\infty(\{\delta \le |t|\le T\}\times M).
\end{equation}

Taking the Fourier transform of the smooth function \eqref{SMOOTH} completes the proof.

\end{proof}

\begin{rem} Alternatively, the Lemma follows from wave front considerations (i.e., (2.6)) and the fact that $L^*(q,\eta) >> T$ on the support
of $\chi_T(q,\eta)$. \end{rem}

\subsection{Terms of type II  for $\delta \leq |t| \leq T$}

\begin{slem} \label{3.4} With the above notation,
\begin{equation}
\left|\frac1{2\pi T}\int \bigl(1-\beta(t)\bigr) \rho(t/T) \, E_B^b (t, q,q) \,
e^{-it\la} \, dt\right|\le C_T\sqrt{\e_0}\, \la^{n-1}+O_{T}(\la^{n-2}).
\end{equation}

\end{slem}

\begin{proof} 

In view of SubLemma \ref{TYPEI} we may insert $(I - \chi_T(q, D))$ to the left of $E_B^b(t,q,q')$. 

We define
$$m_{T,\beta}(\tau)=\frac1{2\pi T} \int
\bigl(1-\beta(t)\bigr) \, \rho(t/T) \, e^{it\tau} \, dt. $$
The  left side of Sublemma \ref{3.4} equals
$$\Bigl|\sum_{j=0}^\infty m_{T,\beta}(\la-\la_j) \, ((I - \chi_T(q, D))  \phi^b_j)(q) \, \phi^b_j(q)\Bigr|.$$
Since $m_{T,\beta}\in {\mathcal S}(\R)$, we can use the Cauchy-Schwarz inequality
to see that for every $N=1,2,3,\dots$ this is dominated by a constant
depending on $T$, $\beta$ and $N$ times
$$\bigl(\sum_{j=0}^\infty (1+|\la-\la_j|)^{-N}|(I - \chi_T(q, D)) \phi^b_j(q)|^2\bigr)^{\frac12} \,
\bigl(\sum_{j=0}^\infty (1+|\la-\la_j|)^{-N}|\phi^b_j(q)|^2\bigr)^{\frac12}.$$
The statement then follows  from \eqref{SMALL},  \eqref{3.3}, 
and in particular from the pointwise cutoff local Weyl laws of Proposition \ref{OpaPTWEYL}   (see  \eqref{z.8}) for the non-glancing part.

\end{proof}

\subsection{Completion of the proof of  Lemma \ref{1.11} and Theorem \ref{SoZthm}.}

 Lemma   \ref{1.11}  follows from Lemma \ref{GLANCINGLOOP}, from  \eqref{rho}, Sublemma \ref{slembeta},   and from Sublemma \ref{3.4}.  As explained at the start of the proof, Lemma \ref{1.11}
implies Theorem \ref{SoZthm}.

\subsection{Tauberian theorem and conclusion of the proof of  Proposition \ref{EPLEM}.}

To prove Proposition \ref{EPLEM} we apply the Tauberian theorem  \ref{ITT} to the three terms, one the glancing part and two from the hyperbolic part; there
should be four terms from the hyperbolic part, but as above we neglect the `off-diagonal' ones since they are smaller.   We give the details because the purpose now is
to get good remainder estimates.  As above, let  $\beta \in C_0^{\infty}(\R)$, $\beta =1$ for $|t| \leq 2 \delta$,
$\beta = 0$ for $|t| \geq 1$ and $\beta_T(t) = \beta(t/T).$ Let $$\left\{ \begin{array}{l} e_I (\lambda) = \Pi^{\leq \epsilon, b}_{[0, \lambda]}(q,q) = \sum_{j: \lambda_j \leq \lambda}  |\chi e_j(q)|^2 \\ \\
e_{II}(\lambda) =  (\chi_T(q, D)  \Pi^{\geq \epsilon, b}_{[0, \lambda]} (\chi_T(q, D)(q,q),\\ \\
e_{III} (\lambda) =( (I - \chi_T(q, D))  ) \Pi^{\geq \epsilon, b}_{[0, \lambda]} (I - \chi_T(q, D))  (q,q). \end{array} \right. $$
All three terms depend on $(\epsilon, T)$. We use the previous notation  $\epsilon_0, $ resp. $ \epsilon$ for  the  cutoff angles to  the
loopset of length $\leq T$, resp. the glancing set.

For each of the three spectral functions $e(\lambda)$, there exist asymptotic expansions of $\rho * d e(\lambda)$ of
the form
 $$\int_0^{\infty} \rho_T(\lambda - \mu) d e(\mu) = a_0 \lambda^{d-1}
+ a_1 (d-1) \lambda^{d-2} + o(\lambda^{d-2}).  $$
For the hyperbolic terms $e_{II}, e_{III}$, the expansions are essentially standard pointwise asymptotics in \cite{I80,HeZ,HHHZ},
since the wave group is a standard Fourier integral operator in the hyperbolic region. The new expansion of
this article (following \cite{M}) is for the glancing term $e_I$, given in  Lemmas\;\ref{SF2}, \ref{rhoSaLEM} \;and \; \ref{GLANCINGLOOP}. By the Tauberian theorem \ref{ITT}, a non-decreasing function satisfying $|e(\lambda) | \leq \lambda^d$. 
Each satisfies
$$|e(\lambda) - a_0 \lambda^d - a_1 \lambda^{d-1}| \leq \frac{C a_0}{T} \lambda^{d-1} + o(\lambda^{d-1}). $$
where the estimates of $\frac{a_0}{T}$ are as follows
$$\left\{ \begin{array}{ll} \rho * d e_I (\lambda) : a_0  \leq C \epsilon, & \rm{Lemmas}\;\ref{SF2}, \ref{rhoSaLEM} \;and \; \ref{GLANCINGLOOP}  \\ \\
\rho_T * e_{II}(\lambda): a_0  \leq C_T\sqrt{\e_0}\, & \rm{SubLemma}\; \ref{3.4}\\ \\
\rho_T * \e_{III} (\lambda):  \leq  CT^{-1}, & \rm{SubLemma} \; \ref{slembeta}. \end{array} \right. $$
These esimates immediately imply Propositions \ref{PTWEYL} and \ref{EPLEM}.

 \end{proof}

\subsection{\label{REMARKON} A remark}
We make a remark on assumption \cite[(1.3)]{M}. 

 Melrose assumes in  \cite[(1.3)]{M} that for $t \not= 0$ the generalized broken geodesic flow $\Phi^t$
fixes no points of $S^* \partial M$, i.e. tangential directions to $\partial M$.
That is, he assumes there are no geodesic loops starting in a tangential direction to the boundary which return tangentially.  He notes  that it is a generic property
of manifolds with concave boundary. But examples exist, such as  complements of polar caps on spheres, since there  exist closed geodesics which touch the boundary tangentially.

This assumption is only used in \cite[Section 4]{M} to eliminate possible
singularities of the wave trace $e(t) = \rm{Tr} \cos t \sqrt{-\Delta}$ due to  closed
geodesics which intersect the boundary tangentially. The assumption eliminates the need to analyze the form of the  wave trace singularities 
due to such diffractive geodesics.

The assumption is not necessary in this article, because we do not need to calculate the
coefficients of the singularities at glancing loops. We  only use
the glancing part of the parametrix to prove normality at $t =0$ of
the singularity and to obtain upper bounds.  In  Theorem \ref{SoZthm} we cut out all closed loops,
including any that might touch the boundary tangentially at beginning
and end.

\section{\label{EXAMPLESECT} Saturating example: Hemisphere of $S^n$}
We now construct examples on the Hemisphere $S_+^n$ of the unit
sphere in $\R^{n+1}$ which saturate the Cauchy data sup norm bounds.
Let  $\sigma (x) = (x', x_n) \to (x', - x_n)$ denote the isometric involution
of $S^n$ through the equator.   to get even/odd eigenfunctions. 

Let  $Z_N^q$ denote the ($L^2$ normalized)  zonal spherical harmonic 
on $S^n$ of degree $N$
with pole at $q \in \partial S_+^n$. Up to $L^2$ normalization 
$Z_N^q = \Pi_N(\cdot, q)$, where $\Pi_N: L^2(S^n) \to \hcal_N$ is
the orthogonal projection to the spherical harmonics of degree N.  To obtain a  Neumann / Dirichlet  eigenfunction
we average it relative to $\sigma$ to get 
$$\Phi^q_N : = \left\{ \begin{array}{ll}  \half (Z_N^q(x', x_n) + Z_N^q(x', - x_n)),
 & \rm{Neumann} \\ &\\ 
 \half (Z_N^q(x', x_n) - Z_N^q(x', - x_n)), & \rm{Dirichlet} \end{array} \right. $$
 
  To see that $\Phi_q^N$ extremizes the sup norm in the Neumann  case, it suffices to observe that  its  value at $q$
is the same as that of $Z_N^q$. Indeed,  up to $L^2$ normalization, 
 $\Phi_q^N$ equals 
$\half(\Pi_N(x, q) + \Pi_N(\sigma(x), q))$ in the Neumann case.  The latter is well-known to saturate
the sup-norm bound (see \cite{SZ2} for background).

The Dirichlet case is more complicated since the normal derivative of 
$\Phi_N^q$ vanishes at $q$. Indeed, the normal derivative
is the same as its $\partial_{x_n}$-derivative at the equator $\partial S_+^n$
and vanishes at $q$ since $Z_N^q$ has a critical point at $q$. However,
we can see that $\partial_{\nu} \Phi_N^q(q')$ asymptotically saturates the sup-norm bound for certain
$q' \in \partial S_+^n$ at a distance $\frac{1}{N}$ from $q$.

As noted above,  $\Phi_N^q(x', x_n)$ is the  $L^2$ normalized  orthogonal 
projection $\Pi_N(x, q)$. As is well-known, the latter is  a certain Legendre polynomial  $P^n_N(\langle x, q \rangle)$ of the inner product  $\langle x, q \rangle$
in the ambient $\R^{n+1}$.  We have $\langle q, q' \rangle  = \cos r (q, q')$
where $r(q,q')$ is the distance in $S^n$.  Specifically,  $P^n_N(t)$ is given by the  Rodrigues' formula and solves the Legendre equation,
$$(1 - t^2) P_N'' + (1 - n) t P_N' + N(N + n -2) P_N = 0. $$
Here, $t = \cos r$ and the pole corresponds to $t = 1$. As the picture below illustrates when the
dimension $n =2$, $P_N(x)$ takes its maximum at $t = 1$ and takes
its maximal slope within $\frac{1}{N}$ of $t = 1$. The technical complication 
is that one must multiply the slope of the Legendre function by the derivative
of its argument $\langle x, q \rangle$.

Along $\partial S_+^n$,  $\partial_{\nu_y} \Phi^q_N(y) = \partial_{x_n} \Phi_N^q(y', 0)$. Since $\Phi^q$ is a radial function in $r(q, \cdot)$,
$\partial_{x_n} \Phi_N^q(y', 0) = \partial_r  P_N(\cos r)$ along $\partial S_+^n$.
Since
$\frac{\partial}{\partial r} f(\cos r) = f'(\cos r) \sin r = f'(t) \sqrt{1 - t^2}$,
$$\begin{array}{lll} 2 \partial_{\nu_q} \Phi^q_N(x)  & = &  \partial_{x_n} [P_N^n(\langle q, x \rangle
- P^n_N(\langle q, \sigma(x) \rangle] \\&&\\ &&= P_N'(\langle q, x \rangle) \partial_{x_n}
\langle q, x \rangle - P_N'(\langle q, \sigma(x) \rangle) \partial_{x_n}
\langle q, \sigma(x) \rangle\\ &&\\ & = &  2 P_N'(\langle q, (x', 0) \rangle
\partial_{x_n} \langle q, (x', x_n) \rangle = 2 P_N'(t) \sqrt{1 - t^2}. 
\end{array} $$
Thus, we need to show that $ P_N'(t) \sqrt{1 - t^2}$ times the $L^2$ normalizing constant of $P_N$ saturates the sup norm bound $N^{1 + \frac{n-1}{2}}$ at a point $t_N$ close to $1$. The normalizing constant
is given by:
$$||P_N^n||_{L^2}^2 = \frac{\omega_{n-1}}{\omega_{n-2}} \frac{1}{m(n, N)}$$
where $m(n, N)$ is the dimension of the space of spherical harmonics
of degree $N$ in dimension $n$ $\simeq N^{n-1}$, and $\omega_m$
is the surface volume of $S^m$.

To prove the saturation bound, we use classical identities for Legendre functions. One has
the recursion relation,

$$P^{n+2}_{N-1}(t) = \frac{(n-1)/2}{(- N) (N + n-2)} \frac{d}{dt} P^n_N(t)
\iff \frac{d}{dt} P^n_N(t) =  \frac{(- N) (N + n-2)}{(n-1)/2}P^{n+2}_{N-1}(t). $$

Note that $Z^q_N = \frac{P_N^n}{||P_N^n||} = C_n \sqrt{m(n,N)} P_N^n.$
All told, $$\begin{array}{lll} \partial_{\nu} \Phi^q_N(q', 0) & = &
 C_n \sqrt{m(n,N)} \frac{d}{dt} P_N^n (t) \sqrt{1 - t^2} \\ &&\\
 & = & C_n'  \sqrt{m(n,N)} \frac{(- N) (N + n-2)}{(n-1)/2}P^{n+2}_{N-1}(t)
 \sqrt{1 - t^2}\\&&\\
 & = &  C_n' \frac{ \sqrt{m(n,N)}}{\sqrt{m(n+2, N-1)}}  \frac{(- N) (N + n-2)}{(n-1)/2} Z^{n+2}_{N -1}(\cos r) \sin r . \end{array}$$
 Note that $\frac{ \sqrt{m(n,N)}}{\sqrt{m(n+2, N-1)}}  \simeq N^{-1}$
 and $ \frac{(- N) (N + n-2)}{(n-1)/2} \simeq N^2$ so  this
 is $ \geq N Z_{N-1}^{n+2}(\cos r) \sin r. $

 It is well-known that 
 $||Z_{N}^{n}||_{\infty} \simeq N^{\frac{n - 1}{2}}, $
and so
 $$||Z_{N-1}^{n+2}||_{\infty} \simeq N^{\frac{n + 1}{2}} \simeq N  ||Z_{N}^{n}||_{\infty}. $$
 But we need to deal with the factor of $\sin r$.

We now choose  $q' = q_N$  so that 
\begin{itemize}

\item (i)\;
 $\frac{C}{N} \leq r(q, q_N) \leq \frac{1}{C N},$
 
  \item(ii) \; $  P_{N-1}^{n+2}(t_N) \geq C  N^{\frac{n + 1}{2}}. $
  
  \end{itemize}
The first condition (i) implies that 
  $ \sin  r(q, q_N) \geq \frac{C}{ N}$ and so
  $$\partial_{\nu} \Phi^q_N(q_N, 0) \geq    C_n'  P^{n+2}_{N -1}(t_N), \; t_N = \cos r(q, q_N) \simeq 1 - (\frac{C}{N})^2. $$
  To complete the proof, we show that there exists $C >0$ so that also
  (ii) holds. Indeed, there is a uniform lower bound for $P_{N-1}^{n+2}$  of order $ N^{\frac{n + 1}{2}}$ on any interval of length $\frac{C}{N^2}$ around $t =1$ and therefore on any interval of the form $[1 - \frac{C}{N^2},1-
  \frac{\epsilon}{N^2}]$ where $C > \epsilon.$ On such an interval
  both (i) and (ii) hold,  implying the desired result,
 $$|\partial_{\nu} \Phi^q_N(q_N, 0)|  \simeq N  N^{\frac{n - 1}{2}}.$$


\section{\label{AIRYAPP} Appendix on Airy functions }

Here we recall the basic definitions and facts regarding Airy functions, referring to \cite{MT} for background.

$$Ai(z) = \frac{1}{2 \pi i} \int_L e^{v^3/3 - z v} dv, $$
where $L$ is any contour that beings at a point at infinity in the sector $- \pi/2 \leq \arg (v) \leq - \pi/6$ and ends
at infinity in the sector $\pi/6 \leq \arg(v) \leq \pi/2$.  In the region $|\arg z| \leq (1 - \delta) \pi$
in $\C - \{\R_-\}$ write $v = z^{\half} + i t ^{\half}$ on the upper half of L and $v = z^{\half} - i t^{\half}$
in the lower half. Then
$$Ai(z) = \Psi(z) e^{- \frac{2}{3} z^{3/2}} $$
with
$$\Psi(z) \sim z^{-1/4} \sum_{j = 0}^{\infty} a_j z^{- 3j/2}, \;\; a_0 = \frac{1}{4} \pi^{-3/2}. $$

Set
$$A_{\pm}(s) = Ai(e^{\pm \frac{2 \pi i}{3} } s). $$

Let 
\begin{equation} \label{Phipm} \Phi_{\pm} (z) = \frac{A'_{\pm}(z)}{A_{\pm}(z)}. \end{equation}
One has 
$$A_{\pm}'(z) Ai(z) - Ai'(z) A_{\pm}(z) = c_{\pm} $$
hence
$$\Phi_{\pm} - \Phi i(z) = c_{\pm} [A_{\pm}(z) Ai(z) ] ^{-1}. $$
The Airy quotients satisfy the nonlinear ODE
$$\Phi'(z) = z - \Phi(z)^2,  \;\; \Phi = \Phi i (z), \mbox{or}\; \Phi_{\pm}(z). $$
Also
$$\Phi_{\pm}(z) = \omega^{\mp 2} \Phi i (\omega^{\mp 2}(z)). $$ 
The poles of $\Phi_{+}$ lie on $e^{- i \pi/3} [-s_0, \infty]$ in the fourth quadrant ($s_0 < 0$). The poles of $\Phi_-(z)$
lie on $e^{i \pi/3} [-s_0, \infty]$ in the first quadrant. Outside any conic neighborhood of these rays,
$$\Phi_{\pm}(z) \sim z^{\half} \sum_{j = 0}^{\infty} b_j^{\pm} z^{- 3j/2}, \;\;\; |z| \to \infty. $$
One has
$$\Phi_+(z) = \overline{\Phi_-(\bar{z})}, \;\;\; b_0^{\pm} = 1. $$
It follows that the Airy quotient $\Phi_{\pm}$ is a classical symbol of order $\half$. 
$$| D^j \frac{A_{\pm}'}{A_{\pm}}(\zeta)| \leq C_j (1 + |\zeta|)^{\half - j}, $$
and
$$\frac{A_{\pm}'}{A_{\pm}}(\zeta) \sim \sum_{j \geq 0} a_j^{\pm} \zeta^{\half - \frac{3 j}{2}}, \;\; \Re \zeta \to \infty. $$
As stated in Section \ref{ORDERsect}, if we substitute $\beta_0$ for
$\zeta$ then $\Phi_{\pm}(\beta_0) \in S^{\frac{1}{3}}_{\frac{1}{3}, 0}.$  

Note that the asymptotic expansions for $a_G, b_G$ in Section 
\ref{ORDERsect} involve $\Phi_{\pm}^{-1} = \frac{A_{\pm}}{A'_{\pm}}$,
and its composition with $\beta_0$ has order $-\frac{1}{3}.$

\section{\label{TAUBERAPP} Appendix on Tauberian theorems}

A model Fourier Tauberian theorem from \cite[Lemma 17.5.6]{HoI-IV} is the following:

\begin{lemma}\label{HorTauber}  Suppose that $\mu$ is a non-decreasing temperate function satisfying
$\mu(0)=0$ and that $\nu$ is a function of locally bounded
variation such that $\nu(0)=0$.  Suppose  that $\phi\in {\scal}({\Bbb R})$ is a fixed positive function
satisfying $\int \phi(\lambda)d\lambda =1$ and $\hat \phi(t)=0$,
$t\notin [-1,1]$.  If
$\phi_a(\tau)=a^{-1}\phi(\frac{\tau}{a})$,
$0<\sigma\le\sigma_0$, assume that for $\lambda\in {\Bbb R}$
\begin{equation}\label{A}
|d\nu(\tau)|\le
M_0 (|\tau| + a_0)^{n-1} d \tau,
\end{equation}
and that
\begin{equation}\label{B}
|((d\mu-d\nu)*\phi_a)(\tau)|\le M_1 (a_1 +|\tau|)^{\kappa},
\end{equation}
for some $a_0, a_1 \geq a$ and  $\kappa \in [0, n-1].$
Then
\begin{equation}\label{C}|\mu(\lambda)-\nu(\lambda)|\le C_m\left( \,
M_0 a (1+|\lambda|)^{n-1} +M_1 (a+|\lambda|)  (a_1+|\lambda|)^{\kappa} \right)
\end{equation}
where $C_m$ is a uniform constant.
\end{lemma}

It implies the following version stated in Ivrii \cite{I80}.

\begin{theorem} \label{ITT} Let $\beta \in C_0^{\infty}(\R)$, $\beta =1$ for $|t| \leq \half$,
$\beta = 0$ for $|t| \geq 1$ and $\beta_T(t) = \beta(t/T).$ Let $e(\lambda)$
be a non-decreasing function satisfying $|e(\lambda) | \leq \lambda^d$. Then
if $$\int_0^{\infty} \hat{\beta}_T(\lambda - \mu) d e(\mu) = a_0 \lambda^{d-1}
+ a_1 (d-1) \lambda^{d-2} + o(\lambda^{d-2}), $$
then
$$|e(\lambda) - a_0 \lambda^d - a_1 \lambda^{d-1}| \leq \frac{C a_0}{T} \lambda^{d-1} + o(\lambda^{d-1}),$$
where $C$ depends only on the dimension $d$ and on $\beta$.
\end{theorem}






\begin{thebibliography}{HHHH}

\bibitem[BSS]{BSS} A. B\"acker, S. F\"ursterberger,  R. Schubert and F. Steiner, 
Behaviour of boundary functions for quantum billiards. J. Phys. A 35 (2002), no. 48, 10293-10310.







\bibitem[Ch]{Ch} J. Chazarain,
Construction de la param\'etrix du probl\`eme mixte hyperbolique
pour l'\'equation des ondes.  C. R. Acad. Sci. Paris S\'er. A-B 276
(1973), A1213--A1215.

\bibitem[Ch2]{Ch2} J. Chazarain, Param\'etrix du probl\`eme mixte
pour l'\' equation des ondes \`a l'int\'erieur d'un domaine
convexe pour les bicaract\'eristiques, Journ\'ees \` Equations aux
d\'eriv\'ees partielles (1975), p. 165-181.


\bibitem[CM]{CM} N. Chernov and R. Markarian, {\it Chaotic billiards}, Mathematical Surveys and Monographs 127 (2006) AMS, Providence, RI.

\bibitem[CTZ12]{ctz}
H.~Christianson, J.~A. Toth, and S.~Zelditch.
\newblock Quantum ergodic restriction for cauchy data: Interior {QUE} and
  restricted {QUE}.
\newblock {\em  Math. Res. Lett.} 20 (2013), no. 3, 465-475. (arXiv:1205.0286).



\bibitem[F]{F} M. Farris,  Egorov's theorem on a manifold with diffractive boundary. Comm. Partial Differential Equations 6 (1981), no. 6, 651- 687; also online at http://scholarship.rice.edu/handle/1911/15545. 


\bibitem[Fr]{Fr} F. G. Friedlander, The Wave Equation on a Curved Space-Time (Cambridge University Press, 1976).

\bibitem[Fr76]{Fr76} F.G. Friedlander, 
The wave front set of the solution of a simple initial-boundary value problem with glancing rays. 
Math. Proc. Cambridge Philos. Soc. 79 (1976), no. 1, 145-159.

\bibitem[FM77]{FM77} F.G.  Friedlander and R.B. Melrose, The wave front set of the solution of a simple initial-boundary value problem with glancing rays. II. Math. Proc. Cambridge Philos. Soc. 81 (1977), no. 1, 97-120.

\bibitem[Gal]{Gal} J. Galkowski, Distribution of Resonances in Scattering by Thin Barriers. To appear in Memoirs of the AMS.  (arXiv  1404.3709).

\bibitem[Gr]{Gr} D. Grieser, Uniform bounds for eigenfunctions of the Laplacian on manifolds with boundary.
 Comm. Partial Differential Equations 27 (2002), no. 7-8,
 1283--1299.








\bibitem[HHHZ]{HHHZ} X. Han, A.  Hassell, H.  Hezari, and S.  Zelditch, Completeness of boundary traces of eigenfunctions,  Proc. Lond. Math. Soc. (3) 111 (2015), no. 3, 749-773.   






\bibitem[HZ04]{HZ}
A. Hassell and S. Zelditch,
\newblock Quantum ergodicity of boundary values of eigenfunctions.
\newblock {\em Comm. Math. Phys.}, 248(1):119--168, 2004.

\bibitem[HZ12]{HeZ}
H. Hezari and S. Zelditch,
\newblock {$C^\infty$} spectral rigidity of the ellipse.
\newblock {\em Anal. PDE}, 5(5):1105--1132, 2012.

\bibitem[H{\"o}r90]{HoI-IV}
L. H{\"o}rmander,
\newblock {\em The analysis of linear partial differential operators. {I-IV}}.
\newblock Springer Study Edition. Springer-Verlag, Berlin, second edition,
  1990.










\bibitem[I80]{I80} V. Ivrii,  The second term of the spectral asymptotics for a Laplace-Beltrami operator on manifolds with boundary. (Russian) Funktsional. Anal. i Prilozhen. 14 (1980), no. 2, 25-34.




\bibitem[JZ]{JZ} J. Jung and S. Zelditch, Number of nodal domains of eigenfunctions of negatively curved surfaces with concave boundary,  Math. Ann. 364 (2016), no. 3-4, 813-840. 

\bibitem[LK59]{LK59} B.R. Levy and J. B.  Keller,  Diffraction by a smooth object. Comm. Pure Appl. Math. 12 1959 159-209.

\bibitem[KR]{KR}  S. Klainerman and I. Rodnianski,  A Kirchoff-Sobolev parametrix for the wave equation and applications,  J. Hyperbolic Differ. Equ. 4 (2007), no. 3, 401-433.

\bibitem[L67]{L67} D. Ludwig,
Uniform asymptotic expansion of the field scattered by a convex object at high frequencies.
Comm. Pure Appl. Math. 20 1967 103-138. 






\bibitem[M]{M} R.B. Melrose, 
Weyl's conjecture for manifolds with concave boundary. Geometry of the Laplace operator , pp. 257-274,
Proc. Sympos. Pure Math., XXXVI, Amer. Math. Soc., Providence, R.I., 1980. 





\bibitem[M75]{M75} R. Melrose, Microlocal parametrices for diffractive boundary value problems, Duke Math. J. 42 (1975), 605-635.



\bibitem[M78]{M78} R. B. Melrose, 
Airy operators. 
Comm. Partial Differential Equations 3 (1978), no. 1, 1-76.


\bibitem[M76]{M76} Melrose, R. B. Equivalence of glancing hypersurfaces. Invent. Math. 37 (1976), no. 3, 165-191.

\bibitem[M84]{M84} R.B. Melrose, The trace of the wave group. Microlocal analysis (Boulder, Colo., 1983), 127-167, Contemp. Math., 27, Amer. Math. Soc., Providence, RI, 1984.


\bibitem[MT]{MT} R. B. Melrose and M. E. Taylor, Boundary problems for wave
equations with grazing and gliding rays, (online at http://www.unc.edu/math/Faculty/met/glide.pdf).



\bibitem[MSj]{MSj} R. B. Melrose and J. Sj\"ostrand,  Singularities of boundary value problems. I. Comm. Pure Appl. Math. 31 (1978), no. 5, 593--617.

\bibitem[Mi]{Mi} Y. Miyazaki, Asymptotic behavior of normal derivatives of eigenfunctions for the Dirichlet Laplacian. J. Math. Anal. Appl. 388 (2012), no. 1, 205-218.


\bibitem[O79]{O79} S. Ozawa, 
Remarks on Hadamard's variation of eigenvalues of the Laplacian.
Proc. Japan Acad. Ser. A Math. Sci. 55 (1979), no. 9, 328-333.




\bibitem[O]{O} S. Ozawa,  Asymptotic property of eigenfunction of the Laplacian at the boundary. Osaka J. Math. 30 (1993), no. 2, 303-314.

\bibitem[PS]{PS} V.M.Petkov and L.N.Stoyanov, {\it Geometry of Reflecting Rays
and Inverse Spectral Problems}, John Wiley and Sons, N.Y. (1992).


\bibitem[SV]{SV} Y. G. Safarov and D.  Vassiliev,   {\em The asymptotic distribution of eigenvalues of partial differential
 operators},
Translations of Mathematical Monographs, {\bf 155} American Mathematical Society, Providence, RI,  1997.







\bibitem[Sm]{Sm} H. F. Smith, Spectral cluster estimates for $C^{1,1}$ metrics.
Amer. J. Math. 128 (2006), no. 5, 1069-1103.

\bibitem[SmS95]{SmS95} H. F.  Smith and C.D. Sogge,  On the critical semilinear wave equation outside convex obstacles. J. Amer. Math. Soc. 8 (1995), no. 4, 879-916.


\bibitem[Sob]{Sob} S. L. Sobolev, 
{\it Partial differential equations of mathematical physics. },
 Addison-Wesley Publishing Co., Inc., Reading, Mass.-London 1964.

\bibitem[Sob2]{Sob2} S. L. Sobolev, Methodes nouvelles a resoudre le probleme de Cauchy pour les equations lineaires hyperboliques normales, Mat. Sb. 1(43) (1936) 31-79.


\bibitem[Sog02]{Sog}
C.~D. Sogge.
\newblock Eigenfunction and {B}ochner {R}iesz estimates on manifolds with
  boundary.
\newblock {\em Math. Res. Lett.}, 9(2-3):205--216, 2002.

\bibitem[STZ11]{STZ}
C.~D. Sogge, J.~A. Toth, and S. Zelditch.
\newblock About the blowup of quasimodes on {R}iemannian manifolds.
J. Geom. Anal., 21(1):150--173, 2011.



\bibitem[SZ02]{SZ1}
C.~D. Sogge and S.  Zelditch,
\newblock Riemannian manifolds with maximal eigenfunction growth.
 Duke Math. J., 114(3):387--437, 2002.

\bibitem[SZ]{SZ2}
C.~D. Sogge and S. Zelditch,
\newblock Focal points and sup-norms of eigenfunctions.
 Rev. Mat. Iberoam. 32 (2016), no. 3, 971-994.  









\bibitem[Tay1]{Tay1}  M. E. Taylor,
Grazing rays and reflection of singularities of solutions to wave equations.
Comm. Pure Appl. Math. 29 (1976), no. 1, 1-38.





\bibitem[Tay2]{Tay2} M. E.  Taylor, {\it Pseudodifferential operators}. Princeton Mathematical Series, 34. Princeton University Press, Princeton, N.J., 1981.

\bibitem[Tay3]{Tay3} M. E. Taylor, Diffraction effects in the scattering of waves,  {\it Singularities in boundary value problems} (Proc. NATO Adv. Study Inst., Maratea, 1980), pp. 271-316,
NATO Adv. Study Inst. Ser. C: Math. Phys. Sci., 65, Reidel, Dordrecht-Boston, Mass., 1981 ( online
at http://www.unc.edu/math/Faculty/met/diffscat.pdf. )

\bibitem[TZ09]{TZ2}
J.~A. Toth and S. Zelditch.
\newblock Counting nodal lines which touch the boundary of an analytic domain.
\newblock {\em J. Differential Geom.}, 81(3):649--686, 2009.

\bibitem[TZ13]{tz1}
J.~A. Toth and S. Zelditch.
\newblock Quantum ergodic restriction theorems: manifolds without boundary.
\newblock {\em Geom. Funct. Anal.}, 23(2):715--775, 2013.





\bibitem[Z4]{Z4}  S. Zelditch, 
Billiards and boundary traces of eigenfunctions. (English summary) Journées "\'Equations aux D\'eriv'ees Partielles'', Exp. No. XV, 22 pp., Univ. Nantes, Nantes, 2003.








\bibitem[Zel92]{z}
S. Zelditch.
\newblock Kuznecov sum formulae and {S}zeg{\H o}\ limit formulae on manifolds.
 Comm. Partial Differential Equations, 17(1-2):221--260, 1992.




\bibitem[Zwo12]{Zw}
M.  Zworski.
\newblock {\em Semiclassical analysis}, volume 138 of {\em Graduate Studies in
  Mathematics}.
\newblock American Mathematical Society, Providence, RI, 2012.

\end{thebibliography}
\end{document}